\documentclass{amsart}
\usepackage{amsthm, amsfonts, mathtools, tikz-cd, amssymb}
\usepackage{enumitem, comment}

\usepackage[pagebackref,pdfborder={0 0 0},urlcolor=black]{hyperref}

\usepackage[english]{babel}
\usepackage[utf8]{inputenc}
\usepackage{lmodern}
\usepackage[babel,protrusion=true,expansion=true,tracking=true]{microtype}
\SetTracking{encoding={*}, shape=sc}{20} 
\AddToHook{cmd/tableofcontents/before}{%
    \microtypesetup{protrusion=false}
}
\AddToHook{cmd/tableofcontents/after}{%
    \microtypesetup{protrusion=true}
}
\usepackage{amsrefs}
\usepackage[thmtools-compat]{keytheorems}
\usepackage{nameref}
\usepackage{fnpct}
\usepackage[shortcuts]{extdash}
\usepackage{nth}

\usepackage{needspace}
\newcommand*{\AvoidPageBreak}{\Needspace*{0pt}}

\hypersetup{
    hypertexnames=false,
    linktoc=all
}
\usepackage{zref-clever}
\zcsetup{
    cap,
    nameinlink,
    abbrev=true,
    countertype={
        defn=definition,
        thm=theorem,
        thmalpha=theorem,
        prop=proposition,
        lem=lemma,
        cor=corollary,
        rem=remark,
        conj=conjecture,
        condenumi=item
    }
}
\zcRefTypeSetup{conjecture}{
  Name-sg = Conjecture ,
  name-sg = conjecture ,
  Name-pl = Conjectures ,
  name-pl = conjectures ,
}

\zcRefTypeSetup{question}{
    Name-sg = Question ,
    name-sg = question ,
    Name-pl = Questions ,
    name-pl = questions ,
}

\newcommand{\cref}[2][]{\zcref[#1]{#2}}
\newcommand{\cnumberref}[2][]{\zcref[noname,#1]{#2}}
\newcommand{\Cref}[2][]{\zcref[S,#1]{#2}}
\newcommand{\lcnamecref}[2][]{\zcref[noref,nameinlink=false,cap=false,#1]{#2}}

\makeatletter
\RenewDocumentCommand{\item}{o}{%
  \@inmatherr\item
  \IfNoValueTF{#1}{%
    \@noitemargtrue\@item[\@itemlabel]%
  }{%
    \def\@currentlabel{#1}\@item[#1]\MakeLinkTarget{}%
  }%
}
\makeatother

\newlist{condenum}{enumerate}{1}
\setlist[condenum]{label=\upshape(\roman*)}

\setlist[enumerate]{font=\upshape}
\newlist{refenum}{enumerate}{2}
\zcsetup{
    counterresetby = {
        refenumii  = refenumi ,
        refenumiii = refenumii ,
        refenumiv  = refenumiii ,
    }
}
\newcommand{\RefEnumSetCounter}[1]{%
    \setlist[refenum,1]{label=\upshape(\arabic*),ref=\csname the#1\endcsname \,\upshape(\arabic*)}%
    \setlist[refenum,2]{label=\upshape(\alph*), ref=\csname the#1\endcsname \,\upshape(\arabic{refenumi}-\alph*)}%
}
\newcommand{\RefEnumPatchCounterType}[1]{%
    \expanded{%
        \noexpand\zcsetup{%
            countertype = {
                refenumi   = #1 ,
                refenumii  = #1 ,
                refenumiii = #1 ,
                refenumiv  = #1 ,
            }}%
    }%
}

\ExplSyntaxOn
\NewDocumentCommand \IfRefenumCounter { m m m }{
    \str_case:enTF {#1} {
        {refenumi} {}
        {refenumii} {}
        {refenumiii} {}
        {refenumiv} {}
    }{#2}{#3}
}
\ExplSyntaxOff

\makeatletter
\AddToHook{env/refenum/begin}[refenumPatch]{%
    \IfRefenumCounter{\@currentcounter}{%
    }{%
        \edef\refenum@parentcounter{\@currentcounter}%
        \edef\refenum@parenttype{\zref@getcurrent{zc@type}}%
        \RefEnumSetCounter{\refenum@parentcounter}%
        \RefEnumPatchCounterType{\refenum@parenttype}%
    }
}
\DeclareHookRule{env/refenum/begin}{refenumPatch}{before}{*}
\makeatother

\ExplSyntaxOn
\cs_new_protected:Nn \gaussler_extract_subref:n {
    \__gaussler_extract_last:n #1
}
\cs_new:Nn \__gaussler_extract_last:n {
    \__gaussler_extract_last:w #1 () \q_stop
}
\cs_generate_variant:Nn \gaussler_extract_subref:n {e}
\cs_new:Npn \__gaussler_extract_last:w #1 (#2) #3 \q_stop {
    \tl_if_empty:nTF {#2} {\textbf{??}} {(#2)}
}
\NewDocumentCommand{\localref}{m}{
    \hyperref[#1]{ \textup { \gaussler_extract_subref:e { \getrefnumber{#1} } } }
}
\ExplSyntaxOff

\usepackage{quiver}
\usetikzlibrary{decorations.markings}
\tikzset{double line with arrow/.style args={#1,#2}{decorate,decoration={markings,%
                    mark=at position 0 with {\coordinate (ta-base-1) at (0,1pt);
                            \coordinate (ta-base-2) at (0,-1pt);},
                    mark=at position 1 with {\draw[#1] (ta-base-1) -- (0,1pt);
                            \draw[#2] (ta-base-2) -- (0,-1pt);
                        }}}}
\tikzset{Equal/.style={-,double line with arrow={-,-}}}

\theoremstyle{definition}
\newtheorem{defn}{Definition}[section]
\newtheorem{example}[defn]{Example}
\newtheorem{question}[defn]{Question}

\newtheorem{rem}[defn]{Remark}
\theoremstyle{plain}
\newtheorem{thm}[defn]{Theorem}
\newtheorem{conj}[defn]{Conjecture}
\newtheorem{lem}[defn]{Lemma}
\newtheorem{prop}[defn]{Proposition}
\newtheorem{cor}[defn]{Corollary}

\newtheorem{thmalpha}{Theorem}

\newtheorem*{coralpha}{Corollary}

\numberwithin{equation}{section}

\newcommand{\bbN}{{\mathbb N}}
\newcommand{\bbQ}{{\mathbb Q}}
\newcommand{\bbR}{{\mathbb R}}

\newcommand{\bbZ}{{\mathbb Z}}
\newcommand{\bbF}{{\mathbb F}}

\newcommand{\bfG}{{\mathbf G}}

\newcommand{\calG}{\mathcal{G}}

\newcommand{\calS}{\mathcal{S}}
\newcommand{\calT}{\mathcal{T}}

\newcommand{\calZ}{\mathcal{Z}}
\newcommand{\calQ}{\mathcal{Q}}
\newcommand{\calO}{\mathcal{O}}

\newcommand{\sfB}{\ensuremath\mathsf{B}}

\newcommand{\sfA}{\ensuremath\mathsf{A}}

\newcommand{\sfE}{\ensuremath\mathsf{E}}

\newcommand{\chain}{\operatorname{\mathsf{Ch}}}
\newcommand{\module}{\mathsf{Mod}}
\newcommand{\modGamma}{\module(\coeffZ[\Gamma])}
\newcommand{\modQGamma}{\module(\coeffQ[\Gamma])}

\newcommand{\modQLambda}{\module(\coeffQ[\Lambda])}

\newcommand{\coeffZ}{\calZ}
\newcommand{\coeffQ}{\calQ}
\newcommand{\excat}[1]{\sfE_{#1}}
\newcommand{\Qmap}[1]{{#1}_{\coeffQ}}

\newcommand{\born}{\ensuremath\mathsf{Born}}
\newcommand{\abelian}{\ensuremath\mathsf{Ab}}

\newcommand{\modQ}{\module(\coeffQ)}
\newcommand{\modZ}{\module(\coeffZ)}
\newcommand{\bornQ}{\born(\module(\coeffQ))}
\newcommand{\bornR}{\born(\module(R))}

\makeatletter
\newcommand{\ind@capheight}[2]{%
    \begingroup
        \settoheight{\dimen0}{$#1#2$}
        \settoheight{\dimen4}{$#1M$}
        \dimen4=0.75\dimen4
        \ifdim\dimen0>\dimen4 \dimen0=\dimen4 \fi
        \raisebox{0pt}[\dimen0][0pt]{\smash{$#1#2$}}%
    \endgroup
}

\NewDocumentCommand{\induction}{s E{^_}{{}}}{%
  \ensuremath{\mathsf{Ind}%
    \IfValueT{#2}{%
      ^{\IfBooleanTF{#1}%
          {\mathpalette\ind@capheight{#2}}%
          {#2}%
      }}%
    \IfValueT{#3}{_{#3}}%
  }%
}
\RenewDocumentCommand{\restriction}{s E{^_}{{}}}{%
  \ensuremath{\mathsf{Res}%
    \IfValueT{#2}{%
      ^{\IfBooleanTF{#1}%
          {\mathpalette\ind@capheight{#2}}%
          {#2}%
      }}%
    \IfValueT{#3}{_{#3}}%
  }%
}
\makeatother

\newcommand{\hocat}{\mathsf{K}}
\newcommand{\acyclic}{\mathsf{Ac}}

\newcommand{\injectarr}{\hookrightarrow}
\newcommand{\onto}{\ensuremath\twoheadrightarrow}
\DeclareRobustCommand\longtwoheadrightarrow{\relbar\joinrel\twoheadrightarrow}

\newcommand{\cone}{\operatorname*{cone}}

\newcommand{\colim}{\operatorname*{colim}}

\newcommand{\defq}{\coloneqq}
\newcommand{\id}{\operatorname{id}}
\newcommand{\im}{\operatorname{im}}

\newcommand{\SL}{\operatorname{SL}}

\NewDocumentCommand{\FP}{E{_}{{}}}{%
    \mathrm{FP}\IfNoValueTF{#1}{}{_{\mkern-2mu#1}}%
}
\NewDocumentCommand{\F}{E{_}{{}}}{%
    \mathrm{F}\IfNoValueTF{#1}{}{_{\mkern-2mu#1}}%
}
\newcommand{\FL}{\operatorname{FL}}

\newcommand{\Nth}[1]{%
    \if\relax\detokenize\expandafter{\romannumeral-0#1}\relax
        \nth{#1}
    \else
        \ensuremath{#1}-th%
    \fi%
}

\newcommand{\rank}{\operatorname{rank}}

\newcommand{\CAT}[1]{\ensuremath{\operatorname{CAT}(#1)}}

\makeatletter
\ExplSyntaxOn
\newcommand{\getslant}[1]{\strip@pt\fontdimen1 #1}
\cs_new_protected:Npn \__fk_fill_kern:n #1
  {
    \group_begin:
    \exp_args:Ne \token_if_cs:NTF { \tl_head:n {#1} }
      {
        \exp_args:Ne \token_case_meaning:NnTF { \tl_head:n {#1} }
          {
            \coeffZ { \mkern-3mu }
          }
          { }
          { \mkern-0.5mu }
      }
      {
        \setbox\usefulbox=\hbox{$\m@th\scriptstyle #1$}%
        \dimen@ \getslant\the\scriptfont\symletters \ht\usefulbox%
        \kern-\dimen@%
      }
    #1
    \group_end:
  }
\NewDocumentCommand \fillKern { m } { \__fk_fill_kern:n {#1} }

\ExplSyntaxOff
\makeatother

\NewDocumentCommand{\filling}{E{^_}{{}}}{%
    \operatorname{FV}%
        \IfValueT{#1}{^{#1}}%
        \IfNoValueTF{#2}{}{_{\fillKern{#2}}}%
}
\NewDocumentCommand{\weightedfilling}{E{^_}{{}}}{%
    \operatorname{wFV}%
        \IfValueT{#1}{^{#1}}%
        \IfNoValueTF{#2}{}{_{\fillKern{#2}}}%
}
\NewDocumentCommand{\geomweightedfilling}{E{^_}{{}}}{%
    \geomweight\!\operatorname{FV}%
        \IfValueT{#1}{^{#1}}%
        \IfNoValueTF{#2}{}{_{\fillKern{#2}}}%
}
\NewDocumentCommand{\combfilling}{E{^_}{{}}}{%
    \operatorname{cFV}%
        \IfValueT{#1}{^{#1}}%
        \IfNoValueTF{#2}{}{_{\fillKern{#2}}}%
}

\newcommand{\verythinspace}{%
    \ifmmode%
        \mskip0.5\thinmuskip%
    \else%
        \ifhmode%
            \kern0.08334em%
        \fi%
    \fi%
}
\newcommand{\delimiterSpace}{%
    \mathchoice{\verythinspace}{}{}{}%
}
\usepackage{mleftright}
\makeatletter
\newcommand*{\norm}[1]{%
    \if@display%
        \@norm{\displaystyle}{\delimiterSpace #1 \delimiterSpace}{\Bigl}{\Bigr}%
    \else%
        \@norm{}{\delimiterSpace #1 \delimiterSpace}{\bigl}{\bigr}%
    \fi%
}
\newcommand*{\@norm}[4]{%
    \@autodelimiterpair{#1}{#2}{\lVert}{\rVert}{#3}{#4}%
}
\newcommand*{\abs}[1]{%
    \if@display%
        \mathinner{\@abs{}{#1}{\Bigl}{\Bigr}}%
    \else%
        \mathinner{\@abs{}{#1}{\bigl}{\bigr}}%
    \fi%
}
\newcommand*{\@abs}[4]{%
    \@autodelimiterpair{#1}{#2}{\lvert}{\rvert}{#3}{#4}%
}
\newcommand*{\@autodelimiterpair}[6]{%
    \sbox0{$#1\vcenter{}$}%
    \dimen0=\ht0 %
    \sbox0{$#1#2$}%
    \dimen2=\dimexpr\ht0-\dimen0\relax
    \dimen4=\dimexpr\dp0+\dimen0\relax
    \ifdim\dimen4>\dimen2 %
        \dimen2=\dimen4 %
    \fi
    \dimen4 = \dimen2 %
    \divide\dimen4 by 500 %
    \dimen4=\delimiterfactor\dimen4 %
    \dimen6=\dimexpr 2\dimen2 - \delimitershortfall\relax
    \ifdim\dimen6>\dimen4 %
        \dimen4=\dimen6 %
    \fi
    \sbox0{$#1#5#3#6#4$}%
    \ifdim\dimexpr\ht0+\dp0\relax>\dimen4 %
        \mleft#3#2\mright#4
    \else
        #5#3#2#6#4
    \fi
}
\makeatother
\newbox\usefulbox
\makeatletter
\let\oldoverline\overline
\def\getslant #1{\strip@pt\fontdimen1 #1}

\def\skoverline #1{\mathchoice
 {{\setbox\usefulbox=\hbox{$\m@th\displaystyle #1$}%
    \dimen@ \getslant\the\textfont\symletters \ht\usefulbox
    \divide\dimen@ \tw@ 
    \kern\dimen@ 
    \oldoverline{\kern-\dimen@ \box\usefulbox\kern\dimen@ }\kern-\dimen@ }}
 {{\setbox\usefulbox=\hbox{$\m@th\textstyle #1$}%
    \dimen@ \getslant\the\textfont\symletters \ht\usefulbox
    \divide\dimen@ \tw@ 
    \kern\dimen@ 
    \oldoverline{\kern-\dimen@ \box\usefulbox\kern\dimen@ }\kern-\dimen@ }}
 {{\setbox\usefulbox=\hbox{$\m@th\scriptstyle #1$}%
    \dimen@ \getslant\the\scriptfont\symletters \ht\usefulbox
    \divide\dimen@ \tw@ 
    \kern\dimen@ 
    \oldoverline{\kern-\dimen@ \box\usefulbox\kern\dimen@ }\kern-\dimen@ }}
 {{\setbox\usefulbox=\hbox{$\m@th\scriptscriptstyle #1$}%
    \dimen@ \getslant\the\scriptscriptfont\symletters \ht\usefulbox
    \divide\dimen@ \tw@ 
    \kern\dimen@ 
    \oldoverline{\kern-\dimen@ \box\usefulbox\kern\dimen@ }\kern-\dimen@ }}%
 {}}
\makeatother
\renewcommand{\overline}[1]{\skoverline{#1}}
\makeatletter
\newsavebox{\@brx}
\newcommand{\llangle}[1][]{\savebox{\@brx}{\(\m@th{#1\langle}\)}%
  \mathopen{\copy\@brx\kern-0.5\wd\@brx\usebox{\@brx}}}
\newcommand{\rrangle}[1][]{\savebox{\@brx}{\(\m@th{#1\rangle}\)}%
  \mathclose{\copy\@brx\kern-0.5\wd\@brx\usebox{\@brx}}}
\makeatother

\newcommand{\size}[1]{\abs{#1}}
\newcommand{\placeholder}{\mskip0.5\thinmuskip\_\mskip0.5\thinmuskip}

\newcommand{\commutes}{\text{\Large$\circlearrowleft$}}

\newcommand{\algweight}{w_\mathrm{a}}
\newcommand{\geomweight}{w_\mathrm{g}}

\hypersetup{
    pdfauthor={Roman Sauer, Jannis Weis},%
    pdftitle={Polynomial homological Dehn functions from non-proper actions}
}

\title[Polynomial homological Dehn functions]{Polynomial homological Dehn functions \mbox{from non-proper actions}}

\author{Roman Sauer}
\address{Fakultät für Mathematik, Karlsruher Institut für Technologie, 76131 Karlsruhe, Germany}
\email{roman.sauer@kit.edu}

\author{Jannis Weis}
\address{Fakultät für Mathematik, Karlsruher Institut für Technologie, 76131 Karlsruhe, Germany}
\email{jannis.weis@kit.edu}
\date{\today}


\begin{document}

\begin{abstract}
    We establish a homological algebra framework for proving polynomiality of higher homological Dehn functions of groups. As an application, we show a combination theorem for polynomial Dehn functions, which is reminiscent of a theorem of Brown for finiteness properties.
\end{abstract}

\maketitle


\section{Introduction and statement of results}

The \emph{\Nth{n} homological Dehn function} $\filling^n_{\bbZ, Y}$ of a CW-complex~$Y$ measures the difficulty of filling integral cellular $(n-1)$-cycles in $Y$ by $n$-chains with respect to the natural $\ell^1$-norm on the integral cellular chain complex.
Let $\Gamma$ be a group of type~$\F_\infty$, that is, $\Gamma$ admits a model~$Y$ of its classifying space $E\Gamma$ that has cocompact skeleta.
Up to a natural equivalence relation, the function $\filling^n_{\bbZ, Y}$ does not depend on the choice of~$Y$, and we refer to it as the \emph{\Nth{n} homological Dehn function} of the group~$\Gamma$, denoted by $\filling^n_{\bbZ, \Gamma}$. For an overview of the comparisons of $\filling^n_{\bbZ, Y}$ to its homotopical and Riemannian analogs we refer to~\cite{abrams_etal}*{Section~2.1} and the references therein.
We also consider homological Dehn functions $\filling^n_{\coeffZ, \Gamma}$ for more general coefficients $\coeffZ$ (see~\cref{sec: coefficients}) under the weaker assumption that $\Gamma$ is of type~$\FP_n(\coeffZ)$ (see~\cref{def: general def of homological Dehn function}).

The main questions we deal with in this paper are: Is $\filling^n_{\coeffZ, \Gamma}$ polynomial? How can we prove it?
We introduce a homological-algebra framework that allows quite elegant proofs of polynomiality in various cases. Conceptually, we view the polynomiality of $\filling^i_{\coeffZ, \Gamma}$ in a certain range $2\le i\le n$ as a strong form of property $\FP_n(\coeffZ)$. The motivation for introducing the homological-algebra framework is to make proofs of the polynomiality of $\filling^i_{\coeffZ, \Gamma}$ in a range $2\le i\le n$ similar to proofs of the finiteness property $\FP_n(\coeffZ)$. Before we say more about the framework, we present some applications below.
\subsection{Results and applications}
\cref{thm: intro thm a la Brown} below is~\cref{thm: geometric brown theorem} in the text. The statement in brackets is a consequence of~\cref{cor: exactness from geometric filling}. We regard~\cref{thm: intro thm a la Brown} as
a quantitative version of a theorem of Brown~\cite{brown_finiteness}*{Proposition~1.1} on finiteness properties. The motivation for this paper was to find a framework that makes a proof along Brown's lines possible.

\pagebreak
\begin{thmalpha}\label{thm: intro thm a la Brown}
    Let $\Gamma$ be a finitely generated group. Let $X$ be a $(d-1)$\=/$\coeffZ$\=/acyclic $\Gamma$\=/simplicial complex with cocompact $d$\=/skeleton.

    We assume that the cellular chain complex $C_\ast(X;\coeffZ) \onto \coeffZ$ is $(d-1)$\=/$\excat{\Gamma}$\=/connected in the sense of~\cref{sec: chain complexes in exact category} (for example, this is the case if $X^{(d)}$ is locally finite, $\coeffZ=\bbZ$ and $\filling^i_{\bbZ, X}$ is polynomial for every $2\le i\le d$).

    Further, we assume for each stabilizer $\Gamma_\sigma$ of a $k$-cell $\sigma$ that
    \begin{condenum}
        \item $\Gamma_\sigma$ is of type $\FP_{d-k}(\coeffZ)$;
        \item $\filling_{\coeffZ,\Gamma_\sigma}^i$ is polynomial for $2 \le i \leq d - k$;
        \item $\Gamma_\sigma$ is finitely generated and at most polynomially distorted in $\Gamma$.
    \end{condenum}
    Then $\filling_{\coeffZ,\Gamma}^i$ is polynomial for $2 \le i \leq d$.
\end{thmalpha}

The notion of $\Gamma$-simplicial complex follows the convention of~\cite{tomdieck}*{II.1}. This means that if a group element maps a simplex to itself, then it acts as the identity on the simplex.

In combination with~\cref{cor: CAT(0)} we obtain the following.

\begin{coralpha}
    Let $X$ be a locally finite \CAT{0}-simplicial complex on which a finitely generated group~$\Gamma$ acts simplicially and cocompactly.
    We assume for each stabilizer $\Gamma_\sigma$ of a $k$-simplex $\sigma$ that
    \begin{condenum}
        \item $\filling_{\bbZ,\Gamma_\sigma}^i$ is polynomial for $2 \le i \leq d - k$;
        \item $\Gamma_\sigma$ is finitely generated and at most polynomially distorted in $\Gamma$.
    \end{condenum}
    Then $\filling_{\bbZ,\Gamma}^i$ is polynomial for $2 \le i \leq d$.
\end{coralpha}

An important example of $\CAT{0}$-simplicial complexes with group actions are Bass-Serre trees, which are rarely locally finite. In this case, however, we can drop the assumption of being locally finite.
One obtains as a direct consequence of~\cref{thm: intro thm a la Brown}, or~\cref{thm: geometric brown theorem} respectively, and~\cref{thm: bass serre tree exact} the following.

\begin{thmalpha}
    Let $d\ge 2$ be an integer. Let $\Gamma$ be the fundamental group of a graph of groups with finitely generated vertex and edge groups such that
    \begin{condenum}
        \item for each vertex group $\Gamma_v$ the function $\filling_{\coeffZ,\Gamma_v}^i$ is polynomial for $2 \le i \leq d$;
        \item for each edge group $\Gamma_e$ the function $\filling_{\coeffZ,\Gamma_e}^i$ is polynomial for $2 \le i \leq d-1$;
        \item all vertex and edge groups are at most polynomially distorted in~$\Gamma$.
    \end{condenum}
    Then $\filling^i_{\coeffZ, \Gamma}$ is polynomial for $2\le i\le d$.
\end{thmalpha}

An assumption on polynomial distortion in the previous theorems is necessary. If, for example, $A$ is a finitely generated subgroup of a finitely presented group~$B$ that is exponentially distorted, then $A$ is exponentially distorted in the double $B\ast_A B$, and $\filling^2_{\bbZ, B\ast_A B}$ is at least exponential. See~\cite{li+manin}*{Lemma~3.6}, which is the homological analog of~\cite{bridson+haefliger}*{Theorem~6.20 on p.~507}.

Another application concerns group extensions. We call an extension $N \hookrightarrow G \onto Q$ \emph{polynomial} if $N$ is at most polynomially distorted in $G$. For such extensions, polynomiality of homological Dehn functions satisfies two-out-of-three properties: it passes from kernel and quotient to the extension, and from kernel and extension to the quotient. The precise statements are~\cref{thm: extension outer to inner,thm: extension left to right}.

\pagebreak
\begin{thmalpha}
    Let $N \hookrightarrow G \onto Q$ be a polynomial extension and let $n \geq 2$.
    \begin{condenum}
        \item If $N$ and $Q$ are of type $\FP_n(\coeffZ)$, and
              $\filling_{\coeffZ,N}^i$ and $\filling_{\coeffZ,Q}^i$ are polynomial for $2 \leq i \leq n$, then $\filling_{\coeffZ,G}^n$ is polynomial.
        \item If $N$ is of type $\FP_{n-1}(\coeffZ)$, $Q$ is of type $\FP_n(\coeffZ)$, $\filling_{\coeffZ,N}^i$ is polynomial for $2 \leq i \leq n - 1$, and $\filling_{\coeffZ,G}^i$ is polynomial for $2 \leq i \leq n$, then $\filling_{\coeffZ,Q}^n$ is polynomial.
    \end{condenum}
\end{thmalpha}

Homological Dehn functions of lattices of semisimple groups are well studied~\cites{bux+wortman, bestvina+eskin+wortman, leuzinger+young}.
For uniform lattices, all homological Dehn functions are polynomial because their model spaces, which are products of non-positively curved symmetric spaces and Bruhat-Tits buildings, are \CAT{0}-spaces. Leuzinger-Young (\cref{thm:LY}) answered the question of polynomiality completely for a non-uniform irreducible lattice $\Gamma$ of a semisimple Lie group~$G$ of real rank~$d\ge 3$: The homological Dehn function $\filling^i_{\bbZ, \Gamma}$ is polynomial for $i<d$ and exponential for $i=d$. The analog of the result by Leuzinger-Young is still open for $S$-arithmetic groups (\cref{conj:BEW}). Bestvina-Eskin-Wortman proved that $\filling^i_{\bbZ, \Gamma}$ is polynomial for $i< \size{S}$ for an $S$\=/arithmetic lattice~$\Gamma$ (\cref{thm: BEW}). In the case where the number of inverted primes $\size{S}$ is small, we obtain~\cref{thm:filling functions for S-arithmetic groups} as an improvement over Bestvina-Eskin-Wortman's result. We formulate here a special case, which is~\cref{thm:arith-special-application}.

\begin{thmalpha}
    For every prime~$p$,
    the \Nth{i} filling function $\filling_{\bbZ,\Gamma}^i$ of $\Gamma=\SL_n(\mathbb{Z}[1/p])$ is polynomial for every $i\in\{2,\dots, n-2\}$.
\end{thmalpha}

\subsection{The framework}

We describe the basic idea of the homological framework in the case $\coeffZ=\bbZ$. We define
a bornology on every $\bbQ[\Gamma]$-module, called the \emph{Cayley bornology}, where $\Gamma$ is a finitely generated group with a word metric~$d$.
The Cayley bornology on $\bbQ[\Gamma]$ is induced by a family of norms (\Cref{eq: definition Cayley norm}).
The completion of $\bbQ[\Gamma]$ with respect to this family of norms is a Fr\'echet algebra $\calS(\Gamma)$, which is studied in~\cites{bader+sauer, meyer, ogle}. The relevant homological structure on $\bbZ[\Gamma]$-modules is that of an exact structure~$\excat{\Gamma}$, tailored to capture the polynomiality of homological Dehn functions.
A short exact sequence \[A \overset{\iota}{\longrightarrow} B \overset{p}{\longrightarrow} C\] of $\bbZ[\Gamma]$\=/modules is \emph{admissible} in~$\excat{\Gamma}$ if there is a $\bbZ$\=/linear section $s \colon C \to B$ and a $\bbZ$\=/linear retraction $r \colon B \to A$ such that $s\otimes_\bbZ\bbQ$ and $r\otimes_\bbZ\bbQ$ are bounded with respect to the Cayley bornologies (\cref{def: admissible morphisms}). Consequently, we have a notion of admissible resolutions, which leads to a notion of $\FP_n(\excat{\Gamma})$ for a $\bbZ[\Gamma]$-module (\cref{defn: FP_n(E)}).
The approach is similar for more general coefficients~$\coeffZ$.

The precise translation between polynomiality and the exact structure is given by the following theorem, which is a combination of~\cref{thm: polynomial fillings from FP_n(E),prop: group polynomial equivalence}.

\begin{thmalpha}
    Let $\Gamma$ be a group of type $\FP_n(\coeffZ)$. Then $\filling_{\coeffZ,\Gamma}^i$ is polynomial for all $2 \leq i \leq n$ if and only if
    $\coeffZ$ is of type $\FP_n(\excat{\Gamma})$.
\end{thmalpha}

The theorem makes precise the viewpoint that the polynomiality of homological Dehn functions is
a strong form of finiteness conditions.
It allows using established techniques from homological algebra in situations where the geometry is hard to control.

We now place our work in the context of the existing literature. The family of norms~\cnumberref{eq: definition Cayley norm} first appeared in work of Connes--Moscovici in connection with the Novikov conjecture~\cite{connes+moscovici}*{\S 6}. They introduced the \emph{polynomial cohomology} of a group, which can be interpreted as a certain Ext-term over the Fr\'echet algebra $\calS(\Gamma)$~\cite{meyer-derived}. Polynomial cohomology has been further studied in~\cites{ji, ji+ramsey, ogle, meyer}. The connection between polynomial cohomology and higher homological Dehn functions was first explored by Ji--Ramsey~\cite{ji+ramsey}. They introduced a property of groups with certain finiteness conditions, called \emph{isocohomological}, defined in terms of polynomial cohomology. Theorem~2.6 of~\cite{ji+ramsey} establishes the equivalence between being isocohomological and having polynomial homological Dehn functions. Under the assumption of being isocohomological, Ji and Ramsey construct fillings of integral cycles that are polynomially bounded in the norm of the cycles. However, these fillings need not have integral coefficients and need not have finite support, so they correspond to a substantially different notion of Dehn functions. Moreover, it does not seem possible to deduce polynomiality of homological Dehn functions from cohomological data alone. Our approach instead works with the underlying chain complexes and their relative homological algebra with respect to an exact structure.

\subsection{Acknowledgements}
The first author thanks the Isaac Newton Institute for the support and hospitality during the program \emph{Graphs, groups, operators} when work on this paper was undertaken. This work was supported by EPSRC Grant Number EP/Z000580/1.
The second author was partially supported by the Deutsche Forschungsgemeinschaft (DFG, German Research Foundation) under
project number 281869850. We thank Uri Bader, Kevin Li and Claudio Llosa Isenrich for many helpful conversations on this subject.

\section{Bornological modules over the group ring}\label{sec: bornology}

In this section we introduce a bornology on modules over the group ring of a finitely generated group that is adapted to studying polynomial filling functions.

\subsection{Bornological rings and modules}\label{sec: bornological rings and modules}

We collect basic notions of bornological rings and modules. A nice exposition of these constructions can be found in~\cite{bambozzi}*{Chapter~1}.

A \emph{bornology} on a set $X$ is a collection of subsets of $X$, called \emph{bounded subsets}, such that $X$ is the union of bounded subsets, and such that subsets and finite unions of bounded subsets are bounded. A \emph{bornological space} is a set endowed with a bornology. A map between bornological spaces is \emph{bounded} if it maps bounded subsets to bounded subsets. The category of bornological spaces is complete and cocomplete and the forgetful functor to sets commutes with limits and colimits~\cite{bambozzi}*{Proposition~1.1.6}. In the special case of the limit being a product $X\times Y$ of bornological spaces, the product bornology is the smallest bornology that contains all products $B_1\times B_2$ of bounded subsets $B_1\subset X$ and $B_2\subset Y$.

A \emph{bornological group} $G$ is a bornological space with the structure of a group such that the multiplication $G\times G\to G$ and inverse map $G\to G$ are bounded. A \emph{bornological ring} $R$ is a bornological space~$R$ with the structure of a ring such that the underlying additive group $(R,+)$ is a bornological abelian group and the ring multiplication $R\times R\to R$ is bounded. Similarly, a \emph{bornological module} $M$ over a bornological ring~$R$ is a bornological space with the structure of an $R$\=/module such that $(M,+)$ is a bornological abelian group and the scalar multiplication $R\times M\to M$ is bounded.

\begin{rem}\label{rem: bornological R-modules as a category}
    The category of bornological $R$\=/modules, denoted by~$\bornR$, is quasi-abelian, complete and cocomplete, and the forgetful functor to $R$\=/modules commutes with limits and colimits~\cite{bambozzi}*{Section~1.4}.

    The category~$\bornR$ is not abelian since for a morphism $f\colon A\to B$ in~$\bornR$ that is injective as a group homomorphism the quotient bornology of $f(A)$ could be strictly finer than the subspace bornology, so that $f\colon A\to f(A)$ is not a bounded isomorphism onto $f(A)$ endowed with its subspace bornology.
\end{rem}

\begin{defn}\label{def: canonical bornology}
    Let $R$ be a bornological ring, and let $M$ be an $R$\=/module. The \emph{canonical bornology} on $M$ is the smallest bornology of $M$ such that $M$ becomes a bornological $R$\=/module and every $R$\=/homomorphism $R\to M$ is bounded.
\end{defn}
The family of bornologies with the above requirement is non-empty as it contains the bornology of all subsets. The intersection of bornologies is a bornology. So the canonical bornology indeed exists. For the following lemma, recall that a subset in the direct sum bornology is bounded if it is contained in a finite direct sum and if its projection to each summand is bounded (see~\cite{bambozzi}*{Section~1.2}).

\begin{lem}\label{lem: properties of canonical bornology: quotient}\AvoidPageBreak
    Let $R$ be a bornological ring. Then the following properties hold:
    \begin{refenum}
        \item\label{lem: properties of canonical bornology: direct sum} The canonical bornology on a free $R$\=/module $\bigoplus_{i\in I} R$ is the direct sum bornology.
        \item\label{lem: properties of canonical bornology: smallest bornology} The canonical bornology on an arbitrary $R$\=/module $M$ is the smallest bornology of~$M$ such that every $R$\=/linear map from a free $R$\=/module to $M$ is bounded.
        \item\label{lem: properties of canonical bornology: universal property} Let $M$ be an $R$\=/module equipped with the canonical bornology and let $N$ be a bornological $R$\=/module. Let $p\colon F\to M$ be a surjective $R$\=/linear map from a free $R$\=/module~$F$ to~$M$.
              Then a homomorphism $f\colon M\to N$ of abelian groups is bounded if and only if the composition $f \circ p\colon F\to N$ is bounded.
    \end{refenum}
\end{lem}

\begin{proof}
    Ad~\localref{lem: properties of canonical bornology: direct sum}. The direct sum bornology makes $F=\bigoplus_{i\in I}R$ into a bornological $R$\=/module. For the addition this follows from the universal property of the direct sum bornology and the fact that the addition on $R$ is bounded. Similarly, for the scalar multiplication.
    If $R\to F$ is $R$\=/linear, then it factors through a direct summand $\bigoplus_{j\in J}R$ for some finite subset $J\subset I$, and the map $R\to \bigoplus_{j\in J}R$ is bounded because the composition with the projection to each summand is bounded. Hence, $R\to F$ is bounded where $F$ carries the direct sum bornology.
    So the canonical bornology is contained in the direct sum bornology.

    On the other hand, the identity map $F\to F$, where we put the direct sum bornology on the domain and the canonical bornology on the codomain, is bounded because of the universal property of the direct sum bornology and the fact that $R\hookrightarrow F\to F$ is bounded for the inclusion $R\hookrightarrow F$ of each summand. This shows that the two bornologies coincide.

    Ad~\localref{lem: properties of canonical bornology: smallest bornology}. By \localref{lem: properties of canonical bornology: direct sum}, every $R$\=/linear map from a free $R$\=/module to $M$ is bounded with respect to the canonical bornology. Thus, the canonical bornology contains the smallest bornology~$\sfB$ such that every $R$\=/linear map from a free $R$\=/module (with the direct sum or canonical bornology) to $M$ is bounded. For the reverse inclusion, it suffices to show that an $R$\=/module $M$ equipped with~$\sfB$ is a bornological $R$\=/module with respect to~$\sfB$. We only need to show that the addition on $M$ is bounded with respect to~$\sfB$.
    The universal property of~$\sfB$ is such that the addition $M\times M\to M$ is bounded if and only if for all $R$\=/linear maps $F\to M$ and $F'\to M$ from free $R$\=/modules the induced map $F\times F'\to M\times M\to M$ is bounded. But this map can be factored into $F\times F'\to F\oplus F'\to  M$ in an obvious way. Both maps in the factorization are clearly bounded. The verification for the scalar multiplication is similar.

    Ad~\localref{lem: properties of canonical bornology: universal property}. The quotient bornology~$\sfB_q$  on~$M$ induced by~$p$ has the universal property stated in~\localref{lem: properties of canonical bornology: universal property}. Therefore, it suffices to show that~$\sfB_q$ is the canonical bornology on~$M$.

    By~\localref{lem: properties of canonical bornology: smallest bornology}, the map $p \colon F \to M$ is bounded for the canonical bornology on~$M$.
    Hence, $\sfB_q$ is contained in the canonical bornology.

    Conversely, $M$ endowed with~$\sfB_q$ is a bornological $R$\=/module.
    Indeed, if $B,B'\subset M$ are $\sfB_q$\=/bounded, choose bounded subsets $A,A'\subset F$ such that $B\subset p(A)$ and $B'\subset p(A')$.
    Then $B+B'\subset p(A+A')$, and $A+A'$ is bounded in~$F$.
    The same argument with bounded subsets of~$R$ shows that scalar multiplication on~$M$ is bounded.
    Moreover, every $R$\=/linear map $u\colon R\to M$ lifts to an $R$\=/linear map $\widetilde{u}\colon R\to F$, since $R$ is free on one generator and $p$ is surjective.
    By~\localref{lem: properties of canonical bornology: direct sum}, the lift~$\widetilde{u}$ is bounded, hence $u=p\circ \widetilde{u}$ is bounded for~$\sfB_q$.
    Therefore, the canonical bornology on~$M$ is contained in~$\sfB_q$, and the two bornologies coincide.
\end{proof}

The next lemma is an immediate consequence of~\cref{lem: properties of canonical bornology: quotient}.

\begin{lem}\label{lem: properties of canonical bornology: maps}
    Let $R$ be a bornological ring. Every $R$\=/linear map between $R$\=/modules that carry their canonical bornologies is bounded.
\end{lem}

\subsection{The coefficient rings}\label{sec: coefficients}

We introduce the coefficients, which are normed rings, for which we will consider higher Dehn functions.

\begin{defn}
    A \emph{norm} on a ring $R$ is a function $\abs{\placeholder} \colon R \to \bbR_{\geq 0}$ such that
    \begin{condenum}
        \item $\abs{x} \geq 0$ and $\abs{x} = 0$ if and only if $x = 0$;
        \item $\abs{x + y} \leq \abs{x} + \abs{y}$;
        \item $\abs{xy} \leq \abs{x}\abs{y}$.
    \end{condenum}
    A ring together with a choice of norm is called a \emph{normed ring}.
    Two norms $\abs{\placeholder}_1, \abs{\placeholder}_2$ are said to be \emph{equivalent} if there exists a constant $C > 0$ such that
    \[
        \frac{1}{C} \cdot \abs{x}_2 \leq \abs{x}_1 \leq C \cdot \abs{x}_2
    \]
    for all $x \in R$.
    A norm on a ring is called \emph{symmetric} if $\abs{{-x}} = \abs{x}$ for all $x \in R$.
    It is called \emph{$\varepsilon$\=/separated} for $\varepsilon > 0$ if $\abs{x} \geq \varepsilon$ for all $x \in R \setminus \{0\}$.
\end{defn}

Every norm is equivalent to a symmetric norm~\cite{weis}*{Lemma 2.4} and if it is $\varepsilon$\=/separated
for some $\varepsilon > 0$ it is equivalent to a $1$\=/separated norm~\cite{weis}*{Lemma 2.5}.
Notable examples of  $1$-separated rings are the integers with the Euclidean norm and any ring equipped with the discrete norm.

\begin{defn}\label{def: coefficients}
    A \emph{normed coefficient system} is a pair $(\coeffZ,\coeffQ)$  of symmetric normed rings such that
    \begin{condenum}
        \item\label{def: coefficients: flat} $\coeffZ \subseteq \coeffQ$ is a flat ring extension;
        \item\label{def: coefficients: restriction} the norm on $\coeffQ$ restricts to the norm on $\coeffZ$;
        \item\label{def: coefficients: separated} $\coeffZ$ is $1$\=/separated;
        \item\label{def: coefficients: density} there exists a constant $C > 1$ such that for every $y \in \bbR_{>0}$ there exists
              some $q \in \coeffQ$ with $C^{-1}y \leq \abs{q} \leq Cy$.
    \end{condenum}
    In particular, \cnumberref{def: coefficients: density} implies that $\coeffQ$ cannot be $\varepsilon$\=/separated for any $\varepsilon > 0$.
\end{defn}

\begin{example}\AvoidPageBreak\leavevmode
    \begin{enumerate}
        \item The prototypical normed coefficient system is $(\bbZ, \bbQ)$ equipped with the Euclidean norm.
        \item For a finite field $\bbF_q$ we consider the norm $\abs{x} \coloneqq q^{-v_{\infty}(x)}$ on $\bbF_q(t)$ induced
              by the valuation $v_{\infty}(\frac{f}{g}) = \deg(g) - \deg(f)$. It restricts to the discrete norm
              on $\bbF_q$. Then $(\bbF_q, \bbF_q(t))$ is a normed coefficient system.
        \item Any $1$\=/separated ring $R$ equipped with the discrete norm admits a normed coefficient system $(R, R')$.
              Let $p > 1$ be arbitrary and consider the ring $R' = R[t,t^{-1}]$. We can equip it with
              the $\ell_1$\=/norm induced by requiring $\abs{t^k} = p^k$ for $k \in \bbZ$.
              Then $R'$ is a flat (even free) ring extension of $R$ satisfying~\cnumberref{def: coefficients: density}
              with $C = p$.
    \end{enumerate}
\end{example}

Throughout the rest of the article we fix a normed coefficient system $(\coeffZ, \coeffQ)$.
One can simply think of $(\bbZ,\bbQ)$ while reading.

\subsection{The group ring as a bornological ring}\label{sec: cayley bornology}

We introduce a bornology on the group ring $\coeffQ[\Gamma]$ of a finitely generated group~$\Gamma$ that turns it into a bornological ring. We then consider the canonical bornology on modules over~$\coeffQ[\Gamma]$, which we will refer to as the \emph{Cayley bornology}.

In the sequel, let $\Gamma$ be a finitely generated group with a word metric~$d_\Gamma$ and associated word length $\ell_\Gamma(\gamma)=d_\Gamma(\gamma,e)$.

On the group ring, we define the following family of norms.
\begin{equation}\label{eq: definition Cayley norm}
    \norm{\sum_{\gamma\in \Gamma} a_\gamma\gamma}^\Gamma_n=\sum_{\gamma\in \Gamma} \abs{a_\gamma}\cdot \bigl(1+\ell_\Gamma(\gamma)\bigr)^n.
\end{equation}

A subset of $\coeffQ[\Gamma]$ is \emph{bounded} if it is bounded with respect to $\norm{\placeholder}^\Gamma_n$ for every $n\in\bbN_0$. The bounded subsets of $\coeffQ[\Gamma]$ form a bornology, which we call the \emph{Cayley bornology}.

Note that $\coeffZ[\Gamma]$ and $\coeffQ[\Gamma]$ are normed rings when equipped with $\norm{\placeholder}^\Gamma_n$.

\begin{rem}\label{rem:characterisation-bounded}
    A $\coeffQ$\=/linear homomorphism $f\colon \coeffQ[\Gamma]\to\coeffQ[\Gamma]$, which is not necessarily $\coeffQ[\Gamma]$\=/linear, is bounded if and only if for every $n\in\bbN_0$ there is a constant $C>0$ and some $m\in\bbN_0$ such that
    \begin{equation}\label{eq:condition-bounded}
        \norm{f(x)}^\Gamma_n\le C\cdot \norm{x}^\Gamma_m
    \end{equation}
    for all $x\in\coeffQ[\Gamma]$.
    To see the \emph{only if}-direction, suppose that $f$ is bounded and there is some $n\in\bbN$ and a sequence
    $x_i\in\coeffQ[\Gamma]$
    with $\norm{f(x_i)}^\Gamma_n>i \norm{x_i}^\Gamma_i$.
    There are $q_i \in \coeffQ$ such that
    $\frac{1}{C \norm{x_i}^\Gamma_i} \leq \abs{q_i} \leq  \frac{C}{\norm{x_i}^\Gamma_i}$,
    where~$C > 1$ is the constant from~\cref{def: coefficients}~\cnumberref{def: coefficients: density}.
    Then $\norm{q_i x_i}^\Gamma_i = \abs{q_i} \norm{x_i}^\Gamma_i \leq C$ and
    $\norm{f(q_i x_i)}^\Gamma_n = \abs{q_i} \cdot \norm{f(x_i)}^\Gamma_n \geq i \cdot C^{-1}$,
    hence
    \[ B\defq \bigl\{ q_i \cdot x_i\mid i\in\bbN_0\bigr\}\]
    is a bounded subset such that $f(B)$ is not bounded.
    Because $\norm{\placeholder}^\Gamma_n \leq \norm{\placeholder}^\Gamma_m$ for  $m \geq n$,
    one also sees that it suffices to check \cref{eq:condition-bounded} for $n \geq n_0$ starting from some $n_0 \in \bbN_0$.
    We note this here because in some places it is convenient to assume that
    $n \geq 1$.
\end{rem}

\begin{prop}\label{prop: group ring is bornological}
    The group ring $\coeffQ[\Gamma]$ equipped with the Cayley bornology is a bornological ring.
\end{prop}
\begin{proof}
    It is quite obvious that the addition is bounded.
    We only prove that the scalar multiplication is bounded. For $k=\min\{n,m\}$, the boundedness is a consequence of the following estimate.
    \begin{align*}
        \norm{\Bigl(\sum_{\gamma\in\Gamma} a_\gamma\gamma\Bigr)\Bigl(\sum_{\lambda\in\Gamma}b_\lambda\lambda\Bigr)}^\Gamma_k
         &= \sum_{\gamma,\lambda\in\Gamma}\abs{a_\gamma b_\lambda}\bigl(1+\ell_\Gamma(\gamma\lambda)\bigr)^k \\
         &\le \sum_{\gamma,\lambda\in\Gamma}\abs{a_\gamma} \abs{b_\lambda}\bigl(1+\ell_\Gamma(\gamma)+\ell_\Gamma(\lambda)\bigr)^k \\
         &\le 2^k\cdot\sum_{\gamma,\lambda\in\Gamma}\abs{a_\gamma} \abs{b_\lambda}\bigl(1+\ell_\Gamma(\gamma)\bigr)^k\bigl(1+\ell_\Gamma(\lambda)\bigr)^k \\
         &\le 2^k\cdot\norm{\sum_{\gamma\in\Gamma}a_\gamma\gamma}^\Gamma_n\cdot \norm{\sum_{\lambda\in\Gamma}b_\lambda\lambda}^\Gamma_m.\qedhere
    \end{align*}
\end{proof}

\begin{lem}\label{lem: integral boundedness}
    Let $f \colon \coeffQ[\Gamma] \to \coeffQ[\Gamma]$ be a $\coeffQ$\=/linear homomorphism. Assume that $f$ restricts to a $\coeffZ$\=/linear
    map $f \colon \coeffZ[\Gamma] \to \coeffZ[\Gamma]$. Then $f$ is bounded if and only if there exists $k\in\bbN$ and $C>0$ such that $\norm{f(\gamma)}_1^\Gamma \leq C\cdot \norm{\gamma}^\Gamma_k$ for every $\gamma \in \Gamma$.
\end{lem}
\begin{proof}
    As $\coeffZ$ is $1$\=/separated, we have $\norm{f(x)}^\Gamma_n \leq \bigl(\norm{f(x)}_1^\Gamma\bigr)^n$ and
    $\bigl(\norm{\gamma}_s^\Gamma\bigr)^k=\norm{\gamma}_{ks}^\Gamma$ for every~$x \in \coeffZ[\Gamma]$, every~$\gamma\in\Gamma$ and all $n, k, s \in \bbN_0$.
    If $f$ is bounded, then by \cref{rem:characterisation-bounded} there are $k \in \bbN$ and $C>0$ such that
    $\norm{f(x)}^\Gamma_1 \le C\norm{x}^\Gamma_k$ for all $x\in \coeffQ[\Gamma]$. Applying this to $x=\gamma$ gives the desired estimate.

    Conversely, assume that there are $k \in \bbN$ and $C > 0$ such that $\norm{f(\gamma)}_1^\Gamma \leq C\cdot \norm{\gamma}^\Gamma_k$ for every~$\gamma \in \Gamma$.
    Then for an arbitrary element $\sum_\gamma a_\gamma \gamma \in \coeffQ[\Gamma]$ we have
    \begin{align*}
        \norm{f\bigl(\sum_\gamma a_\gamma \gamma\bigr)}^\Gamma_n
        \leq \sum_\gamma \abs{a_\gamma} \norm{f(\gamma)}^\Gamma_n
        \leq \sum_\gamma \abs{a_\gamma} \bigl(C\norm{\gamma}^\Gamma_k\bigr)^n
        = C^n \norm{\sum_\gamma a_\gamma \gamma}^\Gamma_{nk},
    \end{align*}
    hence $f$ is bounded.
\end{proof}

\begin{defn}
    The \emph{Cayley bornology} on a $\coeffQ[\Gamma]$\=/module $M$ is the canonical bornology coming from $\coeffQ[\Gamma]$ as a bornological ring with respect to its Cayley bornology.
\end{defn}
We record the following special case of~\cref{lem: properties of canonical bornology: maps}.

\begin{lem}\label{lem: every equivariant map is bounded}
    Every $\coeffQ[\Gamma]$\=/linear map is bounded with respect to the Cayley bornology.
\end{lem}

The following example describes the Cayley bornology on permutation modules in more concrete terms.

\begin{example}\label{exa: bornology permutation modules}
    Let $M = \coeffQ[\Gamma/\Lambda]$ be a $\coeffQ[\Gamma]$\=/permutation module, where $\Lambda \leq \Gamma$ is a subgroup.
    On $M$ we can define the following family of norms.
    \[\norm{\sum_{\gamma\Lambda} a_\gamma \gamma\Lambda}^{\Gamma/\Lambda}_n=\sum_{\gamma\Lambda} \abs{a_\gamma}\cdot (1+ \min\{\ell_\Gamma(\gamma') \mid \gamma' \in \gamma\Lambda\})^n.\]
    They induce a bornology on $\coeffQ[\Gamma/\Lambda]$ where subsets are bounded if they are bounded with respect to each $\norm{\placeholder}^{\Gamma/\Lambda}_n$. For the purpose of this example we refer to this bornology as the \emph{permutation bornology} and to its bounded subsets as \emph{permutation bounded} subsets. We use the description of the canonical bornology, or Cayley bornology in this case, from~\cref{lem: properties of canonical bornology: quotient}.
    Let $B \subset \coeffQ[\Gamma/\Lambda]$ be permutation bounded. Because $\Gamma$ is finitely generated, the projection
    $p \colon \coeffQ[\Gamma] \to \coeffQ[\Gamma/\Lambda]$ admits a section $s \colon \coeffQ[\Gamma/\Lambda] \to \coeffQ[\Gamma]$, such that
    $\norm{x}^{\Gamma/\Lambda}_n = \norm{s(x)}^\Gamma_n$ for all $x \in \coeffQ[\Gamma/\Lambda]$ and all $n \in \bbN_0$.
    Then $s(B)$ is bounded, so $B = p(s(B))$ is bounded.
    Conversely, $p$ is norm-decreasing and therefore $p$ maps bounded subsets to permutation bounded subsets. Since every map from a free module to $\coeffQ[\Gamma/\Lambda]$ factors over $p$, we conclude that the Cayley bornology on $\coeffQ[\Gamma/\Lambda]$ is contained in the permutation bornology. Thus, the two bornologies coincide.

    If $M$ is a free module, then $\Lambda$ is trivial and $\norm{\placeholder}^{\Gamma/\Lambda}_n = \norm{\placeholder}^\Gamma_n$ agrees with the norms from~\cref{eq: definition Cayley norm}.
    For a general permutation module $\coeffQ[S] = \bigoplus_{i \in I} \coeffQ[\Gamma/\Lambda_i]$ we also write $\norm{\placeholder}^\Gamma_n$
    for the $\ell_1$-norm induced by the $\norm{\placeholder}^{\Gamma/\Lambda_i}_n$.
\end{example}

\begin{rem}\label{rem: integral boundedness permutation}
    Let $f \colon \coeffQ[S] \to \coeffQ[T]$ be a $\coeffQ$\=/linear homomorphism between $\coeffQ[\Gamma]$\=/permutation modules.
    Using the description of the bornology from~\cref{exa: bornology permutation modules} one can deduce an analogous statement
    of~\cref{lem: integral boundedness} for permutation modules.
    Namely, assume that $f$ restricts to a $\coeffZ$\=/linear map $\coeffZ[S] \to \coeffZ[T]$.
    Then $f$ is bounded if and only if there exists $k \in \bbN$ and $C > 0$ such that $\norm{f(s)}^\Gamma_1 \leq C \cdot \norm{s}^\Gamma_k$ for
    every $s \in S$.
\end{rem}

\subsection{Functorial properties of the Cayley bornology}

We endow $\coeffQ$ with the standard bornology of subsets that are bounded with respect to the norm. Any bornological module over the bornological ring $\coeffQ[\Gamma]$ with the Cayley bornology is in particular a bornological $\coeffQ$\=/module.
As a consequence of~\cref{lem: every equivariant map is bounded} there is a forgetful functor
\begin{equation}\label{eq: forgetful functor}
    \sfB_\Gamma\colon \modQGamma\to\bornQ
\end{equation}
that maps a $\coeffQ[\Gamma]$\=/module with its Cayley bornology to its underlying bornological $\coeffQ$\=/module.

\begin{lem}\label{lem: coproducts are preserved}
    $\sfB_\Gamma$ preserves all small coproducts.
\end{lem}
\begin{proof}
    Let $\{M_i\}_{i\in I}$ be a collection of $\coeffQ[\Gamma]$\=/modules.
    Let $p \colon F \to \bigoplus_{i \in I} M_i$ be a $\coeffQ[\Gamma]$\=/linear map from a free $\coeffQ[\Gamma]$\=/module.
    By the universal property of the direct sum bornology, the map $p \colon \sfB_\Gamma(F) \to \bigoplus_{i \in I} \sfB_\Gamma(M_i)$
    is bounded if and only if $p \circ \iota$ is bounded for every inclusion $\iota \colon \coeffQ[\Gamma] \to F$
    of a free summand. As the image of $p \circ \iota$ is supported in $\bigoplus_{j \in J} M_j$ for a finite subset $J \subseteq I$,
    $p \circ \iota$ is bounded if and only if $\pi_{M_j} \circ p \circ \iota$ is bounded for all projections $\pi_{M_j}$.
    These are all bounded by the definition of the Cayley bornology, hence every bounded subset in the direct Cayley bornology
    on $\bigoplus_{i \in I} M_i$ is bounded in the direct sum bornology.

    Conversely, consider the inclusion $\iota_j \colon \sfB_\Gamma(M_j) \to \sfB_\Gamma(\bigoplus_{i \in I} M_i)$ of one of the summands.
    By the definition of the Cayley bornology, it is bounded if and only if $\iota_j \circ p$ is bounded for all $\coeffQ[\Gamma]$\=/linear maps
    $p \colon F \to M_j$ where $F$ is a free $\coeffQ[\Gamma]$\=/module. This clearly holds for the Cayley bornology on $\bigoplus_{i \in I} M_i$.
    Therefore, every bounded subset in the direct sum bornology is bounded in the Cayley bornology.
\end{proof}

For finite index sets products and coproducts in $\modQGamma$ and $\bornQ$ agree as $\coeffQ$-modules.
Hence, the above proof also shows that $\sfB_\Gamma$
preserves all finite products. Infinite products however are not preserved by
$\sfB_\Gamma$ as the canonical bornology is the direct sum bornology on such a product and not the product bornology.

\begin{lem}\label{lem: quotients are preserved}
    $\sfB_\Gamma$ preserves quotients.
\end{lem}
\begin{proof}
    Let $q \colon M \onto Q$ be a quotient map in $\modQGamma$.
    Let $p \colon F \to Q$ be a $\coeffQ[\Gamma]$\=/linear map from a free $\coeffQ[\Gamma]$\=/module.
    First we consider the quotient bornology on $Q$, that is, the smallest bornology making $q$ bounded.
    By surjectivity of $q$, there exists a lift $\widetilde{p}\colon F \to M$ such that $q \circ \widetilde{p} = p$.
    Therefore, $p$ is bounded, hence every bounded subset in the Cayley bornology is bounded in the quotient bornology.

    Now consider $Q$ with the Cayley bornology. Then for every $\coeffQ[\Gamma]$\=/linear map $p \colon F \to M$ from a free $\coeffQ[\Gamma]$\=/module
    the composition $q \circ p$ is a $\coeffQ[\Gamma]$\=/linear map from a free module, hence bounded.
    This shows that $q$ is bounded and hence every bounded subset in the quotient bornology is bounded in the Cayley bornology.
\end{proof}

\begin{cor}\label{cor: colimits are preserved}
    $\sfB_\Gamma$ preserves all small colimits.
\end{cor}
\begin{proof}
    In $\modQGamma$ every coequalizer is a quotient, hence by~\cref{lem: quotients are preserved} $\sfB_\Gamma$ preserves coequalizers.
    Together with \cref{lem: coproducts are preserved} this implies that all small colimits are preserved.
\end{proof}

\subsection{Induction of modules}
For the rest of this subsection we consider a subgroup $\Lambda$ of a finitely generated group~$\Gamma$ with a word metric~$d_\Gamma$ and word length $\ell_\Gamma=d_\Gamma(\placeholder, e)$. Further, we fix a set-theoretic map $s\colon \Gamma/\Lambda\to \Gamma$ such that

\begin{equation}\label{eq: section}
    \ell_\Gamma(s(\gamma\Lambda)) = \min\bigl\{ \ell_{\Gamma}(\gamma^\prime) \mid \gamma^\prime\Lambda = \gamma\Lambda \bigr\}.
\end{equation}

There is an induction functor $\induction_\Lambda^\Gamma\colon\modQLambda\to\modQGamma$ that sends $M$ to $\coeffQ[\Gamma]\otimes_{\coeffQ[\Lambda]} M$. Depending on the set-theoretic section $s\colon \Gamma/\Lambda\to\Gamma$, one defines the functorial isomorphism of $\coeffQ$\=/modules
\begin{equation}\label{eq: def nat isomorphism}
    \varphi_M\colon \coeffQ[\Gamma]\otimes_{\coeffQ[\Lambda]} M\cong \coeffQ[\Gamma/\Lambda]\otimes_\coeffQ M,~\gamma \otimes m\mapsto \gamma\Lambda \otimes s(\gamma\Lambda)^{-1}\gamma m
\end{equation}
for every $\coeffQ[\Lambda]$\=/module~$M$. The inverse is given by $\varphi_M^{-1}(\gamma\Lambda\otimes m)=s(\gamma\Lambda)\otimes m$.

The following definition extends the induction functor from $\coeffQ[\Lambda]$\=/linear maps to $\coeffQ$\=/linear maps, depending on the choice of the section~$s$.
\begin{defn}
    For a $\coeffQ$\=/linear map $f\colon M\to N$ between $\coeffQ[\Lambda]$\=/modules we define the
    $\coeffQ$\=/linear map $\induction_\Lambda^\Gamma(f)\colon \induction_\Lambda^\Gamma(M)\to \induction_\Lambda^\Gamma(N)$ by
    \[\induction_\Lambda^\Gamma(f)=\varphi_N^{-1}\circ\bigl(\id_{\coeffQ[\Gamma/\Lambda]}\otimes f\bigr)\circ \varphi_M.\]
    Explicitly, $\induction_\Lambda^\Gamma(f)(\gamma \otimes m) = s(\gamma\Lambda) \otimes f(s(\gamma\Lambda)^{-1}\gamma m)$ for $\gamma \in \Gamma$ and $m \in M$.
\end{defn}

For the purpose of~\cref{thm: bornological induction,prop: bornological induction exact} a $\coeffQ[\Gamma]$\=/module $N$ is called
\emph{$\sfB_\Gamma$\=/boundedly finitely generated} if there exists a finitely generated projective $\coeffQ[\Gamma]$\=/module together with
a surjective $\coeffQ[\Gamma]$\=/map $p\colon P\to N$ admitting a bounded $\coeffQ$\=/linear section $s\colon N\to P$.

Any finitely generated permutation module is automatically $\sfB_\Gamma$\=/boundedly finitely generated, as every quotient $\Gamma/\Lambda$
admits a set-theoretic word length preserving section. Consequently, every direct summand of a finitely generated permutation
module is also $\sfB_\Gamma$\=/boundedly finitely generated.

\begin{thm}\label{thm: bornological induction}
    Let the subgroup $\Lambda<\Gamma$ be finitely generated and at most polynomially distorted in~$\Gamma$.

    Let $M$ and $N$ be $\coeffQ[\Lambda]$\=/modules such that $N$ is $\sfB_\Lambda$\=/boundedly finitely generated.

    Let $f\colon M\to N$ be a $\coeffQ$\=/linear map that is bounded with respect to the Cayley bornologies of $\coeffQ[\Lambda]$\=/modules.
    Then $\induction_\Lambda^\Gamma(f)\colon \induction_\Lambda^\Gamma(M)\to \induction_\Lambda^\Gamma(N)$ is bounded with respect to the Cayley bornologies of $\coeffQ[\Gamma]$\=/modules.
\end{thm}

\begin{proof}
    By assumption, there are $D\ge 1$ and $d\in\bbN$ such that for every~$\lambda\in\Lambda$
    \[\ell_\Lambda(\lambda)\le D\cdot\ell_\Gamma(\lambda)^d.\]
    By taking a larger constant, we may always assume that $\ell_\Gamma(\lambda)\le D\cdot\ell_\Lambda(\lambda)$ for every $\lambda\in\Lambda$. This implies in particular that
    \begin{equation}\label{eq: distortion norms}
        \norm{x}^\Gamma_n\le D^n\norm{x}^\Lambda_n \text{ and }
        \norm{x}^\Lambda_n\le D^n\cdot \norm{x}^\Gamma_{dn}\text{ for every }x\in \coeffQ[\Lambda]\subset\coeffQ[\Gamma].
    \end{equation}

    We first treat the case $M=N=\coeffQ[\Lambda]$. In this case, $\induction_\Lambda^\Gamma(M)=\coeffQ[\Gamma]$ and $\induction_\Lambda^\Gamma(N)=\coeffQ[\Gamma]$.
    Let $f\colon \coeffQ[\Lambda]\to \coeffQ[\Lambda]$ be a bounded $\coeffQ$\=/linear map. This means that for every $n\in\bbN$ there are $C_n\ge 1$ and $l_n\in\bbN$ such that
    $\norm{f(x)}^\Lambda_n\le C_n\cdot \norm{x}^\Lambda_{l_n}$ for every $x\in\coeffQ[\Lambda]$.
    To show $\induction^\Gamma_\Lambda(f)$ is bounded, it suffices to show that there are $C'_n \geq 1$ and $l'_n \in \bbN$ such that
    $\norm{\induction*_\Lambda^\Gamma(f)(\gamma)}_n^\Gamma \leq C'_n \cdot \norm{\gamma}^\Gamma_{l'_n}$.
    For a general element $x = \sum_\gamma a_\gamma \gamma \in \coeffQ[\Gamma]$ it then follows that
    \begin{align*}
        \norm{\induction*_\Lambda^\Gamma(f)(x)}^\Gamma_n
        = \norm{\sum_\gamma a_\gamma \induction*_\Lambda^\Gamma(f)(\gamma)}^\Gamma_n
         &\leq \sum_\gamma \abs{a_\gamma} \norm{\induction_\Lambda^\Gamma(f)(\gamma)}^\Gamma_n \\
         &\leq \sum_\gamma \abs{a_\gamma} C'_n \norm{\gamma}^\Gamma_{l'_n} = C'_n \cdot \norm{x}^\Gamma_{l'_n}.
    \end{align*}
    We have $\induction_\Lambda^\Gamma(f)(\gamma) = s(\gamma\Lambda)f(s(\gamma\Lambda)^{-1}\gamma)$
    and by~\cref{eq: section}, $\ell_\Gamma(s(\gamma\Lambda)^{-1}\gamma)\le 2\cdot \ell_\Gamma(\gamma)$.
    Using~\cref{eq: distortion norms}, we obtain that
    \begin{align*}
        \norm{f(s(\gamma\Lambda)^{-1}\gamma)}^\Gamma_n
         &\leq D^n \cdot \norm{f(s(\gamma\Lambda)^{-1}\gamma)}^\Lambda_{n} \\
         &\leq C_n D^n \cdot \norm{s(\gamma\Lambda)^{-1}\gamma}^\Lambda_{l_n} \\
         &\leq C_n D^{l_n + n} \cdot \norm{s(\gamma\Lambda)^{-1}\gamma}^\Gamma_{dl_n} \\
         &\leq C_n D^{l_n + n}2^{dl_n} \cdot \norm{\gamma}^\Gamma_{dl_n}.
    \end{align*}
    Finally, assume that $s(\gamma\Lambda)$ is chosen to be the $\ell_\Gamma$-minimizing coset representative in $\gamma\Lambda$. Then for any $\gamma \in \Gamma$ we have
    \begin{align*}
        \norm{\induction*^\Gamma_\Lambda(f)(\gamma)}^\Gamma_n
         &\leq \norm{s(\gamma\Lambda)}^\Gamma_n \norm{f(s(\gamma\Lambda)^{-1}\gamma)}^\Gamma_n \\
         &\leq \norm{\gamma}^\Gamma_n \cdot C_n D^{l_n + n}2^{dl_n} \cdot \norm{\gamma}^\Gamma_{dl_n} \\
         &= C_n D^{l_n + n}2^{dl_n} \cdot \norm{\gamma}^\Gamma_{dl_n + n}.
    \end{align*}

    This shows that $\induction_\Lambda^\Gamma(f)$ is bounded in the case $M=N=\coeffQ[\Lambda]$.
    By~\cref{lem: coproducts are preserved}, the case that $M$ is a free $\coeffQ[\Lambda]$\=/module and $N=\coeffQ[\Lambda]$ immediately follows.
    Similarly, the case of a free $\coeffQ[\Lambda]$\=/module $M$ and finitely generated free $\coeffQ[\Lambda]$\=/module $N$ follows from the universal property of products and the fact that $N$ and $\induction_\Lambda^\Gamma(N)$ are
    products of free modules of rank~$1$.

    Next let $M$ be arbitrary and $N$ a finitely generated free $\coeffQ[\Lambda]$\=/module. Choose a $\coeffQ[\Lambda]$\=/epimorphism $p\colon F \onto M$, where $F$ is a free $\coeffQ[\Lambda]$\=/module.
    Then $\induction_\Lambda^\Gamma(f)$ is bounded if and only if $\induction_\Lambda^\Gamma(f)\circ \induction_\Lambda^\Gamma(p) = \induction_\Lambda^\Gamma(f\circ p)$ is bounded by~\cref{lem: properties of canonical bornology: quotient} -- a case that we just settled.

    Next let $M$ be arbitrary and $N$ be $\sfB_\Lambda$-boundedly finitely generated, that is,
    there is a $\coeffQ[\Lambda]$\=/epimorphism $q\colon F'\onto N$ from a finitely generated projective $\coeffQ[\Lambda]$\=/module~$F'$ to~$N$ with a bounded $\coeffQ$\=/linear section~$s$.
    By adding a summand we may and will assume that $F'$ is free.
    Then $\induction_\Lambda^\Gamma(f)$ is bounded because it is a composition
    $\induction_\Lambda^\Gamma(q)\circ\induction_\Lambda^\Gamma(s\circ f)$ of bounded maps.
\end{proof}

\subsection{Restriction of modules}
Let $p \colon \Lambda \to \Gamma$ be a group homomorphism. Then there is a restriction functor $\restriction^\Gamma_\Lambda \colon \modQGamma\to\modQLambda$ induced by $p$. If $p$ is surjective, \cref{exa: bornology permutation modules} implies that
$\sfB_\Lambda(\restriction^\Gamma_\Lambda(\coeffQ[\Gamma])) \cong \sfB_\Gamma(\coeffQ[\Gamma])$ as bornological modules. By
\cref{cor: colimits are preserved}, this also holds for arbitrary free $\coeffQ[\Gamma]$\=/modules. In fact the following general
statement holds.

\begin{thm}\label{thm: restriction commutes}
    Let $p \colon \Lambda \onto \Gamma$ be a surjective group homomorphism between finitely generated groups. Then the square of functors
    \[\begin{tikzcd}
        \modQGamma\ar[d, "\sfB_\Gamma"']\ar[r, "\restriction_\Lambda^\Gamma"] & \modQLambda\ar[d, "\sfB_\Lambda"]\\
        \born(\modQ)\ar[r, "\id"] & \born(\modQ)
    \end{tikzcd}\]
    commutes up to natural isomorphism.
\end{thm}
\begin{proof}
    Let $M$ be a $\coeffQ[\Gamma]$\=/module and $q \colon F \to M$ be a $\coeffQ[\Gamma]$\=/linear map
    from a free $\coeffQ[\Gamma]$\=/module. Since $\sfB_\Gamma(F) \cong \sfB_\Lambda(\restriction^\Gamma_\Lambda(F))$ the map
    $q \colon \sfB_\Gamma(F) \to \sfB_\Lambda(\restriction^\Gamma_\Lambda(M))$
    is bounded.

    Conversely, every $\coeffQ[\Lambda]$\=/linear map $q' \colon F' \to \restriction^\Gamma_\Lambda(M)$ from a free
    $\coeffQ[\Lambda]$\=/module $F' = \bigoplus_{i \in I} \coeffQ[\Lambda]$ factors through
    $\pi \colon \bigoplus_{i \in I} \coeffQ[\Lambda] \to \bigoplus_{i \in I} \coeffQ[\Gamma] = F'_\Gamma$, that is,
    $q' = q'' \circ \pi$ for some $\coeffQ[\Gamma]$\=/linear map $q'' \colon F'_\Gamma \to M$.
    Again, because $\sfB_\Gamma(F'_\Gamma) \cong \sfB_\Lambda(\restriction^\Gamma_\Lambda(F'_\Gamma))$,
    $\sfB_\Lambda(q') = \sfB_\Gamma(q'') \circ \sfB_\Lambda(\pi)$ hence $\sfB_\Lambda(q')$ is bounded.

    By the universal property of the Cayley bornology, we conclude that $\sfB_\Gamma(M) \cong \sfB_\Lambda(\restriction^\Gamma_\Lambda(M))$
    as bornological $\coeffQ$-modules.
\end{proof}

\section{An exact structure on modules over the group ring}\label{sec: exact structure}
In this section we describe an exact structure on the category of $\coeffZ[\Gamma]$\=/modules. Our main reference on exact categories is~\cite{buehler}.

\subsection{The exact category \texorpdfstring{$\excat{\Gamma}$}{E}}

For convenience, we recall some basic notions about exact categories.

\begin{defn}\label{def: exact category}
    Let $\sfA$ be an additive category. A \emph{kernel-cokernel pair} in $\sfA$ is a pair of composable morphisms $A \overset{\iota}{\to} B \overset{p}{\to} C$
    such that $\iota$ is a kernel of $p$ and $p$ is a cokernel of $\iota$. Let $\sfE$ be a fixed class of kernel-cokernel pairs in $\sfA$.
    A monomorphism $\iota$ is called \emph{admissible} if it appears in a kernel-cokernel pair in $\sfE$. Admissible epimorphisms are defined similarly. The kernel-cokernel pairs in $\sfE$ are also called \emph{admissible short exact sequences}.

    An \emph{exact structure} on $\sfA$ is a collection $\sfE$ of kernel-cokernel pairs closed under isomorphism satisfying the following axioms:
    \begin{enumerate}
        \item[(E0)]\label{def: exact category: identity admissible mono}
              Every identity morphism in $\sfA$ is an admissible monomorphism.
        \item[(E0\ensuremath{^\prime})]\label{def: exact category: identity admissible epi}
              Every identity morphism in $\sfA$ is an admissible epimorphism.
        \item[(E1)]\label{def: exact category: composition admissible mono}
              The class of admissible monomorphisms is closed under composition.
        \item[(E1\ensuremath{^\prime})]\label{def: exact category: composition admissible epi}
              The class of admissible epimorphisms is closed under composition.
        \item[(E2)]\label{def: exact category: pushout admissible mono}
              The pushout of an admissible monomorphism along an arbitrary morphism exists and is an admissible monomorphism.
        \item[(E2\ensuremath{^\prime})]\label{def: exact category: pullback admissible epi}
              The pullback of an admissible epimorphism along an arbitrary morphism exists and is an admissible epimorphism.
    \end{enumerate}
\end{defn}

\begin{defn}
    A morphism $f \colon A \to B$ in an exact category is called \emph{admissible} if it factors as a composition
    $A \onto I \injectarr B$ of an admissible epimorphism followed by an admissible monomorphism.

    A sequence of admissible morphisms
    \[\begin{tikzcd}
        A &[-2ex] &[-2ex] B &[-2ex]&[-2ex] C \\[-1em]
        & I && {I'}
        \arrow["f", from=1-1, to=1-3]
        \arrow[two heads, from=1-1, to=2-2]
        \arrow["{f'}", from=1-3, to=1-5]
        \arrow[two heads, from=1-3, to=2-4]
        \arrow[hook, from=2-2, to=1-3]
        \arrow[hook, from=2-4, to=1-5]
    \end{tikzcd}\]
    is called $\sfE$\=/exact if $I \injectarr B \onto I'$ is a short exact sequence.
    This notion naturally extends to longer sequences of admissible morphisms.

    An \emph{$\sfE$\=/resolution} of an object $A$ is a chain complex $C_\ast$ together with an admissible morphism
    $C_0 \to A$ such that the augmented complex $C_\ast \to A$ is $\sfE$\=/exact.
\end{defn}

Let $f \colon A \to B$ be a $\coeffZ$\=/linear map between $\coeffZ[\Gamma]$\=/modules.
We write $\Qmap{f}$ for the map $\coeffQ \otimes_\coeffZ f$. We say that $f_\coeffQ$ is \emph{bounded} if it is bounded with respect to the Cayley bornology on $\coeffQ \otimes_\coeffZ A$ and $\coeffQ \otimes_\coeffZ B$.
Equivalently, we say that $f$ is \emph{$\sfB_\Gamma$\=/bounded}.

\begin{defn}\label{def: admissible morphisms}
    A short exact sequence $A \overset{\iota}{\to} B \overset{p}{\to} C$ of $\coeffZ[\Gamma]$\=/modules is \emph{admissible} if there is a $\coeffZ$\=/linear section $s \colon C \to B$ and a $\coeffZ$\=/linear retraction $r \colon B \to A$ such that $\Qmap{s}$ and $\Qmap{r}$ are bounded.
    We denote the collection of admissible short exact sequences by $\excat{\Gamma}$.
\end{defn}

\begin{defn}
    An object $P$ in an exact category is called \emph{$\sfE$-projective} if
    the functor $\hom_\sfA(P,\placeholder) \to \abelian$ is exact, that is,
    for every admissible epimorphism $A \onto B$ and every morphism $P \to B$
    there exists a solution to the lifting problem
    \[\begin{tikzcd}
        & P & \\[-0.5em]
        A && B
        \arrow[dashed, from=1-2, to=2-1]
        \arrow[from=1-2, to=2-3]
        \arrow[two heads, from=2-1, to=2-3]
    \end{tikzcd}\]
    making the triangle commute.
\end{defn}

\subsection{Verification of the axioms of an exact category}
It is clear from the definition that $\excat{\Gamma}$ satisfies~\cnumberref{def: exact category: identity admissible mono, def: exact category: identity admissible epi}.

\begin{lem}\label{lem: mono with bounded retraction is admissible}
    Let $f \colon A \to B$ be a monomorphism admitting a $\sfB_\Gamma$\=/bounded $\coeffZ$\=/linear retraction $r \colon B \to A$.
    Then the short exact sequence $A \xrightarrow{f} B \xrightarrow{g} B/\im(f)$ is admissible.
    In particular, $f$ and $g$ are admissible.
\end{lem}
\begin{proof}
    Let $C = B/\im(f)$ and define $s' \colon B \to B$ by $s'(b) \coloneqq b - f(r(b))$. Then $s'(f(a)) = f(a) - f(r(f(a))) = 0$ for every~$a \in A$, hence
    $s'$ descends to a $\coeffZ$\=/linear map $s \colon C \to B$. The map $s$ satisfies $g(s(b + \im(f))) = g(b - f(r(b))) = g(b)$, hence is a section of $g$.
    Because $s$ comes from the universal property of the quotient,
    exactness of $\coeffQ \otimes_\coeffZ -$ and \cref{cor: colimits are preserved} show that $\Qmap{s}$ is bounded in the Cayley bornology.
\end{proof}

\begin{rem}
    Let $0 \to A \xrightarrow{f} B \xrightarrow{g} C \to 0$ be an exact sequence in $\modGamma$.
    The constructed section $s \colon C \to B$ from~\cref{lem: mono with bounded retraction is admissible} satisfies
    $r \circ s = 0$ and $f \circ r + s \circ g = \id_B$.
\end{rem}

\begin{lem}\label{lem: induced bornology and cayley bornology equal for admissible map}
    Let $f \colon M \to N$ be an admissible monomorphism. Then the Cayley bornology on $\coeffQ \otimes_\coeffZ M$ coincides with the subspace bornology induced by $\coeffQ \otimes_\coeffZ f$.
\end{lem}
\begin{proof}
    It always holds that a bounded subset of $\coeffQ \otimes_\coeffZ M$ in the Cayley bornology is also bounded in the induced bornology.
    Conversely, let $B \subseteq \coeffQ \otimes_\coeffZ M$ be bounded in the induced bornology, that is, $\Qmap{f}(B)$ is bounded in the Cayley bornology on
    $\coeffQ \otimes_\coeffZ N$. Let $r \colon N \to M$ be a retraction as in~\cref{def: admissible morphisms}.
    Then $\Qmap{(r \circ f)}(B) = B$ is bounded in the Cayley bornology.
\end{proof}

\begin{prop}[{$\excat{\Gamma}$ satisfies~\cnumberref{def: exact category: pushout admissible mono}}]\label{prop: pushout admissible mono}
    The pushout of an admissible monomorphism along an arbitrary morphism is admissible.
\end{prop}
\begin{proof}
    Let $f \colon A \injectarr B$ be an admissible monomorphism with $\coeffZ$\=/linear retraction $r$ and let $g \colon A \to A'$ be any morphism.
    Then $g \circ r \circ f = g \circ \id = g$.
    As the forgetful functor from $\modQGamma$ to $\modZ$ preserves pushouts, the universal property of $B'$ yields a $\coeffZ$\=/linear
    map $r' \colon B' \to A'$ such that
    \[\begin{tikzcd}
        A & B \\
        {A'} & {B'} \\
        && {A'}
        \arrow["f", hook, from=1-1, to=1-2]
        \arrow["g"', from=1-1, to=2-1]
        \arrow["\Big\ulcorner"{anchor=center, pos=0.125}, draw=none, from=2-2, to=1-1]
        \arrow["{g'}", from=1-2, to=2-2]
        \arrow["{g \circ r}", curve={height=-12pt}, from=1-2, to=3-3]
        \arrow["{f'}"', from=2-1, to=2-2]
        \arrow["\id"', curve={height=12pt}, from=2-1, to=3-3]
        \arrow["{r'}"', dashed, from=2-2, to=3-3]
    \end{tikzcd}\]
    commutes, that is, $r' \circ f' = \id$.
    Moreover, by~\cref{cor: colimits are preserved} applying the functor $\sfB_\Gamma(\coeffQ \otimes_\coeffZ -)$ preserves pushouts, hence $\Qmap{r'}$ is bounded. By~\cref{lem: mono with bounded retraction is admissible}, $f'$ is admissible.
\end{proof}

\begin{prop}[{$\excat{\Gamma}$ satisfies \cnumberref{def: exact category: pullback admissible epi}}\label{prop: pullback admissible epi}]
    The pullback of an admissible epimorphism along an arbitrary morphism is admissible.
\end{prop}
\begin{proof}
    Let $p \colon A \onto B$ be an admissible epimorphism. Let $r \colon A \to \ker(p)$ be a $\coeffZ$\=/linear retraction such that $\Qmap{r}$ is bounded.
    Now consider the following diagram for an arbitrary morphism $f \colon B' \to B$.
    \[\begin{tikzcd}
        {\ker(p')} & {A'} & {B'} \\
        {\ker(p)} & A & B
        \arrow["\iota'", hook, from=1-1, to=1-2]
        \arrow["\varphi", from=1-1, to=2-1]
        \arrow["\cong"', draw=none, from=1-1, to=2-1]
        \arrow["{p'}", two heads, from=1-2, to=1-3]
        \arrow["{f'}"', from=1-2, to=2-2]
        \arrow["\Big\lrcorner"{anchor=center, pos=0.125}, draw=none, from=1-2, to=2-3]
        \arrow["f", from=1-3, to=2-3]
        \arrow["\iota", hook, from=2-1, to=2-2]
        \arrow["p", two heads, from=2-2, to=2-3]
        \arrow["r", curve={height=-10pt}, from=2-2, to=2-1, start anchor={[yshift=-0.25ex,xshift=-0.25ex]south}]
    \end{tikzcd}\]
    There is an equivariant isomorphism $\varphi \colon \ker(p') \to \ker(p)$. Define $r' \colon A' \to \ker(p')$ as $r' = \varphi^{-1} \circ r \circ f'$.
    Then $\Qmap{r'}$ is bounded and
    \[(\varphi^{-1} \circ r \circ f') \circ \iota' = \varphi^{-1} \circ (r \circ \iota) \circ \varphi = \varphi^{-1} \circ \varphi = \id_{\ker(p')}.\]
    Therefore, by~\cref{lem: mono with bounded retraction is admissible} $p'$ is admissible.
\end{proof}

\begin{prop}[{$\excat{\Gamma}$ satisfies \cnumberref{def: exact category: composition admissible mono, def: exact category: composition admissible epi}}]
    The classes of admissible monomorphisms and admissible epimorphisms in $\modGamma$ are closed under composition.
\end{prop}
\begin{proof}
    For admissible monomorphisms this is a direct consequence of~\cref{lem: mono with bounded retraction is admissible}.
    Let $f \colon A \onto B$ and $g \colon B \onto C$ be admissible epimorphisms.
    Again, by~\cref{lem: mono with bounded retraction is admissible}, it suffices to show that $\iota_{gf} \colon \ker(g \circ f) \to A$
    is an admissible monomorphism.
    By assumption, the inclusions $\iota_f$ and $\iota_g$ of the kernels $\ker(f)$ and $\ker(g)$ are admissible, hence admit $\sfB_\Gamma$\=/bounded $\coeffZ$\=/linear retractions $r_f$ and $r_g$ respectively.
    Applying \cref{prop: pullback admissible epi} to the pullback
    \[\begin{tikzcd}
        {\ker(g \circ f)} & {\ker(g)} \\
        A & B
        \arrow["v", two heads, from=1-1, to=1-2]
        \arrow["{\iota_{gf}}"', from=1-1, to=2-1]
        \arrow["\Big\lrcorner"{anchor=center, pos=0.125}, draw=none, from=1-1, to=2-2]
        \arrow["{\iota_g}", from=1-2, to=2-2]
        \arrow["f", two heads, from=2-1, to=2-2]
    \end{tikzcd}\]
    yields that the sequence
    \[
        \ker(f) \overset{u}{\injectarr} \ker(g \circ f) \overset{v}{\onto} \ker(g)
    \]
    is admissible, hence admits a $\coeffZ$\=/linear retraction $r$ of $u$
    and $\coeffZ$\=/linear section $s$ of $v$ such that $\Qmap{r}$ and $\Qmap{s}$
    are bounded. Explicitly we have $r = r_f \circ \iota_{gf}$.
    By~\cref{lem: mono with bounded retraction is admissible}, we may assume that $s$ is chosen such that
    $u \circ r + s \circ v = \id_{\ker(g \circ f)}$.
    Now consider the map $r' \coloneqq u \circ r_f + s \circ r_g \circ f \colon A \to \ker(g \circ f)$.
    It is $\sfB_\Gamma$\=/bounded and satisfies
    \begin{align*}
        r' \circ \iota_{gf}
         &= u \circ r + s \circ r_g \circ (\iota_g \circ v) \\
         &= u \circ r + s \circ v = \id_{\ker(g \circ f)},
    \end{align*}
    hence $\ker(g \circ f) \to A$ is admissible by~\cref{lem: mono with bounded retraction is admissible}.
\end{proof}

The culmination of the preceding results of this section is the following theorem.

\begin{thm}
    $\excat{\Gamma}$ is an exact structure on $\modGamma$.
\end{thm}

\subsection{Finiteness properties in \texorpdfstring{$\excat{\Gamma}$}{E}}

Finiteness properties of modules over the group ring are usually defined in terms of projective resolutions. We adapt this definition to our setting of exact categories.

\begin{defn}\label{defn: FP_n(E)}
    An \emph{$\FP_n(\excat{\Gamma})$\=/resolution} of a $\coeffZ[\Gamma]$\=/module $M$ is a partial $\excat{\Gamma}$\=/resolution
    \[ P_n \to P_{n-1} \to \cdots \to P_0 \onto M \to 0\]
    such that $P_i$ is a finitely generated projective $\coeffZ[\Gamma]$\=/module for $i \leq n$.

    A $\coeffZ[\Gamma]$\=/module is said to be of type $\FP_n(\excat{\Gamma})$ if it admits an $\FP_n(\excat{\Gamma})$\=/resolution.
    Moreover, if we can choose the $P_i$ to be free $\coeffZ[\Gamma]$\=/modules, $M$ is said to be of type $\FL_n(\excat{\Gamma})$.
\end{defn}

Note that we require the modules to be projective as $\coeffZ[\Gamma]$\=/modules instead of the weaker condition of
being $\excat{\Gamma}$\=/projective.
This guarantees that every module appearing in such a partial resolution is a
direct summand of a free module, which we make use of in the proof of~\cref{thm: FP_n equivalent to FL_n}.

The following \lcnamecref{prop: bornological induction exact, prop: bornological restriction exact} are consequences of~\cref{thm: bornological induction,thm: restriction commutes}
and give criteria for boundedness of $\coeffZ$\=/linear maps when passing from subgroups or to quotients.

\begin{prop}\label{prop: bornological induction exact}
    Let $\Lambda < \Gamma$ be a finitely generated subgroup that is at most polynomially distorted in $\Gamma$.
    Let $f \colon M \to N$ be a $\coeffZ$\=/linear map between $\coeffZ[\Lambda]$\=/modules, where $\coeffQ \otimes_\coeffZ N$ is
    $\sfB_\Lambda$\=/boundedly finitely generated\footnote{A sufficient condition is for $N$ to be of type~$\FP_0(\excat{\Lambda})$.}.
    Then $\induction^\Gamma_\Lambda(f)$
    is $\sfB_\Gamma$\=/bounded if $f$ is $\sfB_\Lambda$\=/bounded.
    Moreover, for any admissible exact sequence $\sigma \in \excat{\Lambda}$ where the left and the middle module are of type~$\FP_0(\excat{\Lambda})$, the sequence $\induction^\Gamma_\Lambda(\sigma)$ is
    an admissible exact sequence in $\excat{\Gamma}$.
\end{prop}

\begin{prop}\label{prop: bornological restriction exact}
    Let $p \colon \Lambda \onto \Gamma$ be a surjective group homomorphism between finitely generated groups.
    Let $f \colon M \to N$ be a $\coeffZ$\=/linear map between $\coeffZ[\Gamma]$\=/modules. Then $f$ is $\sfB_\Gamma$\=/bounded if and only
    if $\restriction^\Gamma_\Lambda(f)$ is $\sfB_\Lambda$\=/bounded.
    Moreover, for any admissible exact sequence $\sigma \in \excat{\Gamma}$, the sequence $\restriction^\Gamma_\Lambda(\sigma)$ is
    an admissible exact sequence in $\excat{\Lambda}$.
\end{prop}

The following is a direct consequence of~\cref{prop: bornological induction exact} together with the fact that
$\induction^\Gamma_\Lambda(-)$ preserves finite generation and projectivity.

\begin{prop}\label{prop: FPn induction}
    Let $\Lambda < \Gamma$ be a finitely generated subgroup that is at most polynomially distorted in $\Gamma$.
    If a $\coeffZ[\Lambda]$\=/module $M$ is of type $\FP_n(\excat{\Lambda})$ then
    $\induction^\Gamma_\Lambda(M)$ is of type $\FP_n(\excat{\Gamma})$.
\end{prop}

For $\coeffZ[\Gamma]$\=/modules, the properties $\FP_n$ and $\FL_n$ coincide, by Schanuel's
Lemma~\cite{brown}*{VIII Lemma 4.2}. The following
\lcnamecref{prop: schanuel exact categories} is the analog for exact categories.
The statement in~\cite{mathieu+rosbotham} is formulated for injective objects.
The projective version as stated here is formally
dual (see~\citelist{\cite{buehler}*{11.2}\cite{mathieu+rosbotham}*{p.~11}}).

\begin{prop}[\cite{mathieu+rosbotham}*{Proposition 3.4}]\label{prop: schanuel exact categories}
    Let $M$ be an object in an exact category $(\sfA,\sfE)$. Let
    \[0 \to P_n \to P_{n-1} \to \dots \to P_0 \onto M \to 0\]
    and
    \[0 \to P'_n \to P'_{n-1} \to \dots \to P'_0 \onto M \to 0\]
    be $\sfE$\=/exact sequences of admissible morphisms. Assume that $P_i$ and $P'_i$ are $\sfE$\=/projective for $i \leq n - 1$.
    Then
    \[ P_0 \oplus P'_1 \oplus P_2 \oplus \cdots \cong P_0' \oplus P_1 \oplus P'_2 \oplus \cdots\text{.}\]
\end{prop}

\begin{thm}\label{thm: FP_n equivalent to FL_n}
    Let $M$ be a $\coeffZ[\Gamma]$\=/module. Then $M$ is of type $\FP_n(\excat{\Gamma})$ if and only if it is of type $\FL_n(\excat{\Gamma})$.
\end{thm}
\begin{proof}
    Clearly $\FL_n(\excat{\Gamma})$ implies $\FP_n(\excat{\Gamma})$ for all $n$. For the converse direction
    let $P_\ast \onto M$ be a $\FP_n(\excat{\Gamma})$\=/resolution.
    We inductively build a suitable partial free $\excat{\Gamma}$\=/resolution.
    Because $P_0$ is finitely generated and projective, there exists a finitely generated $\coeffZ[\Gamma]$\=/module $Q$ such that
    $L_0 \coloneqq P_0 \oplus Q$ is free.
    In particular, $Q \injectarr L_0 \onto P_0$ is an admissible short exact sequence.
    The composition $L_0 \onto P_0 \onto M$ is admissible, hence $M$ is of type $\FL_0(\excat{\Gamma})$.
    Now assume that for $k < n$ we have constructed a partial free $\excat{\Gamma}$\=/resolution
    \[L_{k} \to L_{k-1} \to \cdots \to L_0 \onto M \to 0.\]
    Define $K_{k} \coloneqq \ker(P_{k} \to P_{k-1})$ and $K'_{k} \coloneqq \ker(L_{k} \to L_{k-1})$.
    By~\cref{prop: schanuel exact categories}, there is an isomorphism
    \begin{align*}
        A \coloneqq K_k \oplus L_{k-1} \oplus P_{k-2} \oplus \cdots \xrightarrow{\;\varphi\;} B \coloneqq K'_k \oplus P_{k-1} \oplus L_{k-2} \oplus \cdots.
    \end{align*}
    From the argument for $n = 0$ it follows that there exist finitely generated free modules $L'_0, \ldots, L'_k$ and
    admissible epimorphisms $L'_k \onto K_k$, $L'_{k-1} \onto L_{k-1}$, $L'_{k-2} \onto P_{k-2}, \ldots$.
    Consider the composition
    \begin{align*}
        L_{k+1} \coloneqq \bigoplus_{i = 0}^k L'_i \longtwoheadrightarrow A \overset{\varphi}{\longrightarrow} B \longtwoheadrightarrow K'_k
    \end{align*}
    where the last map is the projection onto the $K'_k$ summand. Each map is an admissible epimorphism, hence
    \[L_{k+1} \to L_{k} \to \cdots \to L_0 \onto M \to 0\]
    is a partial $\excat{\Gamma}$\=/resolution consisting of finitely generated free modules.
\end{proof}

\section{Chain complexes in the exact category of \texorpdfstring{$\coeffZ[\Gamma]$}{ZG}-modules}%
\label{sec: chain complexes in exact category}

First we study some homotopy connectivity properties of chain complexes in an exact category. We then apply these results to the exact category $(\modGamma, \excat{\Gamma})$. The main purpose of this section is to show~\cref{thm: blow-up}, which is similar
to~\cite{cheap_rebuilding}*{Proposition~3.3}. The latter can be seen as an explicit version of~\cite{brown2}*{Lemma~1.5}.

\subsection{Chain complexes and mapping cones}
Let $(\sfA, \sfE)$ be an exact category. We denote by $\chain(\sfA)$ the category of chain complexes in~$\sfA$.
Our convention is that
a chain complex is concentrated in non-negative degrees unless said otherwise. A chain complex that is concentrated in degrees greater than or equal to~$-1$ is referred to as an \emph{augmented chain complex}.

Differentials of chain complexes are generically denoted by~$\partial$. For a chain homotopy $h$ between chain maps $f$ and $g$ we write $h\colon f\simeq g$ to indicate that $\partial h_n+h_{n-1}\partial=f_n-g_n$ for every $n\in\bbZ$.

For a chain complex $C_\ast$ we write $C_\ast^{(k)}$ for its $k$\=/skeleton, that is, $C_i^{(k)}$ is $C_i$ for $i\le k$ and $0$ for $i>k$.

The \emph{suspension} of a chain complex $C_\ast$ is the chain complex $\Sigma C_\ast$ with chain modules $\Sigma C_n=C_{n-1}$ and differentials $\partial^{\Sigma C_\ast}_n=-\partial^{C_\ast}_{n-1}$.

The \emph{mapping cone} of a chain map $f\colon C \to D$ is the chain complex~$\cone(f)$ whose \Nth{n} chain module is $C_{n-1}\oplus D_n$
with the differential
\[\partial^f_n(x_{n-1},y_n)=(-\partial^C_{n-1} (x_{n-1}), \partial^D_n (y_n)+f_{n-1}(x_{n-1})).\]

\begin{defn}
    Let $n \in \bbN \cup \{\infty\}$.
    A chain complex $A\in\chain(\sfA)$ is \emph{$n$\=/acyclic with respect to~$\sfE$} or \emph{$n$\=/$\sfE$\=/acyclic} if $A_{n+1} \to A_{n} \to \cdots$ consists of admissible morphisms and is $\sfE$\=/exact.
    It is called \emph{$n$\=/connected with respect to~$\sfE$} or \emph{$n$\=/$\sfE$\=/connected} if it is homotopy equivalent to an $n$\=/$\sfE$\=/acyclic complex.
    A chain map is \emph{$n$\=/connected with respect to~$\sfE$} or \emph{$n$\=/$\sfE$\=/connected} if its mapping cone is $n$\=/$\sfE$\=/connected.
    For $n = \infty$ such a map is also called an \emph{$\sfE$\=/quasi-isomorphism}.
\end{defn}

\begin{rem}
    If $\sfA$ is a quasi-abelian category, the collection of all kernel-cokernel pairs forms an exact structure $\sfE_\mathrm{max}$~\cite{buehler}*{13.2}.
    An $\sfE_\mathrm{max}$-projective module then is a projective module in the classical sense.
    Given a different exact structure $\sfE$ on $\sfA$ it is sometimes convenient to refer to $\sfE_\mathrm{max}$ when talking about
    $n$\=/acyclic and $n$\=/connected chain complexes. In this case we write $n$\=/$\sfA$\=/acyclic and $n$\=/$\sfA$\=/connected when referring
    to the exact structure $\sfE_\mathrm{max}$. If $\sfA = \module(R)$ for some ring $R$, we also write $n$\=/$R$\=/acyclic and $n$\=/$R$\=/connected.
\end{rem}

\begin{rem}
    The categories we are interested in all have a property called idempotent completeness~\cite{buehler}*{Definition 6.1}.
    As in~\cite{buehler}*{10.7,10.9}, one shows that in this case any complex homotopy equivalent to an $n$\=/$\sfE$\=/acyclic complex already has
    to be $n$\=/$\sfE$\=/acyclic.
    In particular, for abelian categories~$\sfA$, a complex is $n$\=/$\sfA$\=/connected if and only if its homology vanishes up to degree $n$ and
    an $n$\=/$\sfA$\=/connected map $f$ is the same as a map for which $H_\ast(f)$ is an isomorphism up to degree $n$ and $H_{n+1}(f)$
    is surjective.
\end{rem}

\begin{rem}\label{rem: FPnE and connectivity}
    Let $M$ be a $\coeffZ[\Gamma]$\=/module. That $\varepsilon \colon P_\ast \onto M$ is an $\FP_n(\excat{\Gamma})$\=/resolution of $M$ is equivalent to the condition that $P_i$ is a finitely generated projective $\coeffZ[\Gamma]$\=/module for $i \leq n$ and $\varepsilon$ is $(n - 1)$\=/$\excat{\Gamma}$\=/connected with respect
    to $\excat{\Gamma}$, where we regard $M$ as a chain complex concentrated in degree~$0$.
\end{rem}

In the sequel the datum of a \emph{homotopy commutative square} of chain maps consists of a square of chain maps
\begin{equation}\label{eq: htp square}
    \begin{tikzcd}
        X\ar[r, "f"]\ar[d, "a" left] & Y\ar[d, "b"]\\
        X'\ar[r, "f'"] & Y'\arrow[from=1-1, to=2-2, phantom, "\commutes_{H}" description]
    \end{tikzcd}
\end{equation}
and a chain homotopy $H\colon f'\circ a\simeq b\circ f$. Mapping cones are functorial with respect to homotopy commutative squares in the following sense. The maps in~\eqref{eq: htp square} induce a chain map
\begin{gather*}
    (a, b;H)\colon \cone(f)\to \cone(f')\\
    (a,b;H)_n(x,y)=\bigl(a_{n-1}(x),b_n(y)-H_{n-1}(x)\bigr).
\end{gather*}

\begin{lem}\label{lem: cone of chain map}
    Let $n \in \bbN \cup \{\infty\}$. Let $f \colon X \to Y$ be a map of chain complexes in $\sfA$ such that
    $X$ is $(n-1)$\=/$\sfE$\=/connected and $Y$ is $n$\=/$\sfE$\=/connected. Then $\cone(f)$ is $n$\=/$\sfE$\=/connected.
\end{lem}
\begin{proof}
    The case where $n = \infty$ is proved in~\citelist{\cite{buehler}*{10.3}\cite{neeman}*{1.1}}.
    For $n < \infty$, one proceeds in exactly the same way while keeping track of the degrees in which the
    complexes $X$ and $Y$ need to be $\sfE$\=/exact.
\end{proof}

Let $\acyclic_n(\sfA)$ denote the full subcategory of the homotopy category $\hocat(\sfA)$ consisting
of $n$\=/$\sfE$\=/connected chain complexes. It is well known that $\hocat(\sfA)$ has the structure of a triangulated
category~\cite{weibel}*{10.2.4}. The direct sum of $n$\=/$\sfE$\=/connected chain complexes is again $n$\=/$\sfE$\=/connected~\cite{buehler}*{2.9},
hence $\acyclic_n(\sfA)$ is a full additive subcategory of $\hocat(\sfA)$.
Clearly it is closed under suspension and the previous \lcnamecref{lem: cone of chain map}
implies that $\acyclic_n(\sfA)$ is closed under forming mapping cones.
However, $\acyclic_n(\sfA)$ is not closed under desuspension unless $n = \infty$,
preventing it from being a triangulated subcategory.
Nevertheless, it is a \emph{suspended category}~\citelist{\cite{tattar}*{2.24}\cite{keller}*{7}}, which is a slightly weaker notion sufficient for our purposes.

The following \lcnamecref{lem: connectivity pushforward, lem: connectivity of composition} are straightforward
consequences of the suspended structure on $\acyclic_n(\sfA)$.
For the convenience of the reader, we provide ad-hoc proofs instead of relying on general theorems in the literature on suspended categories.

\begin{lem}\label{lem: connectivity pushforward}
    Let $n \in \bbN \cup \{\infty\}$. Let $f \colon X \to Y$ be a map of chain complexes in $\sfA$ such that $f$ is $n$\=/$\sfE$\=/connected
    and $X$ is $n$\=/$\sfE$\=/connected. Then $Y$ is $n$\=/$\sfE$\=/connected.
\end{lem}
\begin{proof}
    Let $p \colon \cone(f) \to \Sigma X$ be the projection $p_n(x_{n-1},y_n) = x_{n-1}$ and consider the triangle
    \[
        \Sigma^{-1}\cone(f) \xrightarrow{\Sigma^{-1}p} X \xrightarrow{f} Y \to \cone(f)
    \]
    obtained by rotating the triangle associated to the mapping cone sequence of $f$.
    Then $Y$ is homotopy equivalent to $\cone(\Sigma^{-1}p)$.
    Because $f$ is $n$\=/$\sfE$\=/connected, $\cone(f)$ is $n$\=/$\sfE$\=/connected, hence $\Sigma^{-1}\cone(f)$ is $(n-1)$\=/$\sfE$\=/connected.
    \Cref{lem: cone of chain map} yields that $\cone(\Sigma^{-1}p)$
    is $n$\=/$\sfE$\=/connected, hence $Y$ is also $n$\=/$\sfE$\=/connected.
\end{proof}

\begin{lem}\label{lem: connectivity of composition}
    Let $n \in \bbN \cup \{\infty\}$. Let $f \colon X \to Y$ and $g \colon Y \to Z$ be maps of chain complexes in $\sfA$.
    If $f$ and $g$ are $n$\=/$\sfE$\=/connected, then so is $g \circ f$.
\end{lem}

The proof can be found in~\cite{keller}*{10.1}. For convenience, we include a direct argument.
\begin{proof}
    By the octahedral axiom there is a triangle
    \[
        \cone(f) \xrightarrow{p} \cone(g \circ f) \to \cone(g) \to \Sigma \cone(f)
    \]
    in $\hocat(\sfA)$.
    Because $g$ is $n$\=/$\sfE$\=/connected, so is $p$, as $\cone(p)$ is homotopy
    equivalent to $\cone(g)$.
    Now apply \cref{lem: connectivity pushforward} to conclude $n$\=/$\sfE$\=/connecivity of $g \circ f$.
\end{proof}

\begin{lem}\label{lem: cones connectivity induced map}
    Let $n \in \bbN \cup \{\infty\}$. We consider
    \begin{equation*}
        \begin{tikzcd}
            X\ar[r, "f"]\ar[d, "a" left] & Y\ar[d, "b"]\\
            X'\ar[r, "f'"] & Y'\arrow[from=1-1, to=2-2, phantom, "\commutes_{H}" description]
        \end{tikzcd}
    \end{equation*}
    a homotopy commutative square in $\chain(\sfA)$ such that $a$ is $(n-1)$\=/$\sfE$\=/connected and $b$ is $n$\=/$\sfE$\=/connected.
    Then the induced map $(a, b; H)\colon\cone(f)\to \cone(f')$ is $n$\=/$\sfE$\=/connected.
\end{lem}
\begin{proof}
    For any chain map $\varphi \colon C \to D$ the map $(\id, -\id)$ is an isomorphism from $\cone(-\varphi)$ to $\cone(\varphi)$.
    Therefore, by assumption $\cone(-a)$ is $(n-1)$\=/$\sfE$\=/connected and $\cone(b)$ is $n$\=/$\sfE$\=/connected.
    Now consider the composition
    \[\psi \colon \cone(-a) \xrightarrow{(\id,-\id)} \cone(a) \xrightarrow{(-f,-f';H)} \cone(b)\]
    where~$H$ is a homotopy from $b \circ (-f)$ to $(-f') \circ a$. By~\cref{lem: cone of chain map},
    $\cone(\psi)$ is $n$\=/$\sfE$\=/connected.

    The components of $\cone(a,b;H)$ and $\cone(\psi)$ are given by
    \begin{align*}
        \cone(a,b;H)_i &= X_{i-2} \oplus Y_{i-1} \oplus X'_{i-1} \oplus Y'_{i} \\
        \cone(\psi)_i  &= X_{i-2} \oplus X'_{i-1} \oplus Y_{i-1} \oplus Y'_{i}
    \end{align*}
    and one checks that, up to the isomorphism interchanging $Y_{i-1}$ and $X'_{i-1}$, the differentials
    of $\cone(\psi)$ agree with the differentials of $\cone(a,b;H)$. It follows that $\cone(a,b;H)$,
    and hence the map $(a,b;H)$, is $n$\=/$\sfE$\=/connected.
\end{proof}

\begin{lem}\label{lem:quasi-isomorphism-induce-iso-on-homsets}
    Let $f\colon X\to Y$ be an $\sfE$-quasi-isomorphism and let $P$ be a complex of $\sfE$-projectives. Then postcomposition with
    $f$ induces an isomorphism
    \[
        f_\ast\colon
        \hom_{\hocat(\sfA)}(P,X)
        \xrightarrow{\;\cong\;}
        \hom_{\hocat(\sfA)}(P,Y).
    \]
\end{lem}
\begin{proof}
    Consider the distinguished triangle $X \overset{f}{\to} Y \to C_f \to \Sigma X$, where $C_f$ denotes the mapping cone $\cone(f)$.
    By~\cite{weibel}*{10.2.8}, applying $\hom_{\hocat(\sfA)}(P,-)$ gives an exact sequence
    \[
        \hom_{\hocat(\sfA)}(P,\Sigma^{-1}C_f)
        \to
        \hom_{\hocat(\sfA)}(P,X)
        \xrightarrow{f_\ast}
        \hom_{\hocat(\sfA)}(P,Y)
        \to
        \hom_{\hocat(\sfA)}(P,C_f).
    \]
    Because $f$ is an $\sfE$-quasi-isomorphism, $C_f$ is homotopy equivalent to an acyclic complex.
    Then \cite{buehler}*{Corollary~12.6} implies that the outer two groups vanish.
    Therefore, $f_\ast$ is an isomorphism.
\end{proof}

\begin{lem}\label{lem:lifting-map-up-to-homotopy}
    Let $f \colon X \to X'$ be a chain map.
    Let $p \colon P \to X$ and $p' \colon P' \to X'$ be $\sfE$\=/quasi-isomorphisms such that $P$ is a complex of $\sfE$\=/projectives.
    Then there exist a chain map $\widehat{f} \colon P \to P'$
    and an $\sfE$\=/quasi\=/isomorphism $\cone(\widehat{f}\,) \to \cone(f)$
    fitting into a diagram
    \[\begin{tikzcd}
        P & {P'} & \cone(\widehat{f}\,)\\
        X & {X'} & \cone(f)
        \arrow["{\widehat{f}}", from=1-1, to=1-2]
        \arrow["p"', from=1-1, to=2-1]
        \arrow["{\commutes_H}"{description}, draw=none, from=1-1, to=2-2]
        \arrow["{\commutes}"{description}, draw=none, from=1-2, to=2-3]
        \arrow["{p'}", from=1-2, to=2-2]
        \arrow["f", from=2-1, to=2-2]
        \arrow[from=1-2, to=1-3]
        \arrow[from=2-2, to=2-3]
        \arrow[from=1-3, to=2-3, shift right=5pt]
    \end{tikzcd}\]
    such that the right-hand square is commutative and the left-hand square commutes up to a chain homotopy $H$.
\end{lem}
\begin{proof}
    If $\sfA = \module(R)$ for a ring $R$ this is proved in~\cite{brown2}*{Lemma~1.3}.
    By \cref{lem:quasi-isomorphism-induce-iso-on-homsets} $p'$ induces an isomorphism
    \[
        p'_\ast \colon \hom_{K(\sfA)}(P,P') \xrightarrow{\;\cong\;} \hom_{K(\sfA)}(P,X').
    \]
    Let $\widehat{f}$ be the image of $f \circ p$ under the inverse of $p'_\ast$.
    Then $p' \circ \widehat{f} \simeq f \circ p$ for a chain homotopy $H$.
    By \cref{lem: cones connectivity induced map}, the induced map $(p,p';H) \colon \cone(\widehat{f}\,) \to \cone(f)$ is an $\sfE$\=/quasi\=/isomorphism.
    It is also easily checked that it makes the right-hand square in the diagram above commute.
\end{proof}

\subsection{A combination result}
In the following let $\sfA$ be an additive category equipped with two exact
structures $\sfE_0$, $\sfE_1$ such that $\sfE_1 \subseteq \sfE_0$.
We are mainly interested in the case where $\sfA = \modGamma$, $\sfE_0 = \sfE_{\max}$,
and $\sfE_1 = \excat{\Gamma}$.
The extra structure $\sfE_1$ is needed because $\sfE_1$ might not have enough projectives for $\sfE_1$-resolutions to exist, for example this is likely the case for $\excat{\Gamma}$.

\begin{thm}\label{thm: blow-up}
    Let $n\in\bbN \cup \{\infty\}$ and $X \in \chain(\sfA)$.
    For each~$j \geq 0$, let $\varepsilon^j\colon P^j\onto X_j$ be an $\sfE_0$-projective resolution of $X_j$ such that $\varepsilon^j$ is $(n-j)$\=/$\sfE_1$\=/connected\footnote{The condition is void for~$j>n$.}.

    Then there exists a chain complex~$\widehat{X}$ together with a filtration of chain complexes $(\widehat{X}^{[k]})_{k \geq 0}$ and a
    chain map $q\colon \widehat{X}\to X$ such that for every~$k \geq 0$ the following hold:
    \begin{enumerate}
        \item $\widehat{X}^{[k]}$ is the mapping cone of a chain map~$\Sigma^{k-1}P^k\to \widehat{X}^{[k-1]}$;
        \item $q^k$ is $n$\=/$\sfE_1$\=/connected;
        \item $q^k$ is an $\sfE_0$\=/quasi\=/isomorphism,
    \end{enumerate}
    where $q^k\coloneqq q|_{\widehat{X}^{[k]}}\colon \widehat{X}^{[k]}\to X^{(k)}$ is the restriction to $\widehat{X}^{[k]}$.
    In particular, the chain modules of~$\widehat{X}$ are of the form
    \[
        \widehat{X}_n=\;\bigoplus_{\mathclap{j+i=n}}\; P^j_i\text{.}
    \]
\end{thm}
\begin{proof}
    We proceed by induction on~$k \geq 0$.
    Set $\widehat{X}^{[0]}\coloneqq P^0$ and let~$q^0\colon P^0\to X^{(0)}$ be given by the augmentation $P_0^0\onto X_0$ in degree~$0$. By assumption, $q^0$ is an $n$\=/$\sfE_1$\=/connected $\sfE_0$\=/quasi\=/isomorphism and a degreewise admissible epimorphism.

    For the inductive step, assume that the chain complex $\widehat{X}^{[k-1]}$ and an $n$\=/$\sfE_1$\=/connected $\sfE_0$\=/quasi\=/isomorphism
    $q^{k-1}\colon \widehat{X}^{[k-1]}\to X^{(k-1)}$ have been constructed.

    Consider the following diagram
    \[\begin{tikzcd}
        \Sigma^{k-1}P^k
        \arrow[d, swap, two heads, "\Sigma^{k-1}\varepsilon^k"]
        & \widehat{X}^{[k-1]}
        \arrow[d, two heads, "q^{k-1}"] \\
        \Sigma^{k-1}X_k
        \arrow[r, "\partial^X_k"]
        & X^{(k-1)}
    \end{tikzcd}\]
    where $\varepsilon^k\colon P^k\to X_k$ is given by the augmentation $P^k_0\onto X_k$ in degree~$0$.
    Since $\varepsilon^k$ is $(n-k)$\=/$\sfE_1$\=/connected, $\Sigma^{k-1}\varepsilon^k$ is $(n-1)$\=/$\sfE_1$\=/connected.
    By assumption, the chain complex $\Sigma^{k-1}P^k$ consists of projective modules and the chain map~$q^{k-1}$ is a $\sfE_0$\=/quasi\=/isomorphism.
    Hence by \cref{lem:lifting-map-up-to-homotopy}, there exists a chain map~$\widehat{\partial^X_k}\colon \Sigma^{k-1}P^k\to \widehat{X}^{[k-1]}$ making the square homotopy commutative through a chain homotopy $H$.

    We set $\widehat{X}^{[k]}\coloneqq \cone(\widehat{\partial^X_k})$ and let $q^k\colon \widehat{X}^{[k]}\to X^{(k)} = \cone(\partial_k^X)$ be the induced chain map on mapping cones.
    Thus, by \cref{lem:lifting-map-up-to-homotopy}, we have a homotopy commutative diagram of mapping cone sequences
    \[\begin{tikzcd}
        \Sigma^{k-1}P^k
        \arrow[r, "\widehat{\partial^X_k}"]
        \arrow[d, swap, two heads, "\Sigma^{k-1}\varepsilon^k"]
        \arrow[dr, phantom, "\commutes_{\mathrlap{H}}"]
        & \widehat{X}^{[k-1]}
        \arrow[d, "q^{k-1}"]
        \arrow[r]
        & \widehat{X}^{[k]}
        \arrow[d, two heads, "q^k"] \\
        \Sigma^{k-1}X_k
        \arrow[r, "\partial^X_k"']
        & X^{(k-1)}
        \arrow[r]
        & X^{(k)}.
    \end{tikzcd}\]
    Since the chain maps~$\Sigma^{k-1}\varepsilon^k$ and~$q^{k-1}$ are $(n-1)$\=/$\sfE_1$\=/connected and $n$\=/$\sfE_1$\=/connected, so is~$q^k$
    by~\cref{lem: cones connectivity induced map}.
    Moreover, \cref{lem:lifting-map-up-to-homotopy} also shows that $q^k$ is an $\sfE_0$\=/quasi\=/isomorphism.

    Taking the limit, we set $\widehat{X}\coloneqq \colim_k \widehat{X}^{[k]}$ and $q\coloneqq \colim_k q^{k}$, which are as desired.
    Note that by construction, we have
    \[
        \widehat{X}^{[k]}_n=\;\bigoplus_{\mathclap{\substack{j+i=n,\\ 0 \le j\le k}}}\; P^j_i
    \]
    and $\widehat{X}^{[k]}_n=\widehat{X}_n$ for~$k\ge n$ thus $\widehat{X}_n = \bigoplus_{j+i=n} P^j_i$.
    In particular, we see that $\colim_k \widehat{X}^{[k]}$ is a degreewise finite coproduct, hence exists in any additive category.
\end{proof}

The next two results are specifically for the exact category $(\modGamma, \excat{\Gamma})$.

\begin{prop}[2-out-of-3]\label{prop:2 out of 3}
    Let $0 \to A \xrightarrow{\iota} B \xrightarrow{\pi} C \to 0$ be an admissible short exact sequence of $\coeffZ[\Gamma]$\=/modules and let $n \in \bbN \cup \{\infty\}$.
    \begin{refenum}
        \item\label{prop:2 out of 3:quotient} If $A$ is of type $\FP_{n-1}(\excat{\Gamma})$ and $B$ is of type $\FP_n(\excat{\Gamma})$, then $C$ is of type $\FP_n(\excat{\Gamma})$;
        \item\label{prop:2 out of 3:extension} If $A$ is of type $\FP_{n}(\excat{\Gamma})$ and $C$ is of type $\FP_n(\excat{\Gamma})$, then $B$ is of type $\FP_n(\excat{\Gamma})$;
        \item\label{prop:2 out of 3:kernel} If $B$ is of type $\FP_{n}(\excat{\Gamma})$ and $C$ is of type $\FP_{n+1}(\excat{\Gamma})$, then $A$ is of type $\FP_n(\excat{\Gamma})$.
    \end{refenum}
\end{prop}
\begin{proof}
    Ad~\localref{prop:2 out of 3:quotient}. By assumption, there exist a $\FP_{n-1}(\excat{\Gamma})$\=/resolution of $A$ and a $\FP_{n}(\excat{\Gamma})$\=/resolution
    of $B$. Let $\widehat{X}$ and $q \colon \widehat{X} \to X$ be the chain complex and map obtained from~\cref{thm: blow-up}
    applied to the complex $X_\ast = 0 \to A \injectarr B \to 0$.
    Then $\widehat{X}_i$ is finitely generated for $i \leq n$ and $q$ is $(n-1)$\=/$\excat{\Gamma}$\=/connected.
    Because the short exact sequence is admissible,
    the equivalence $X_\ast \xrightarrow{\pi} C$ is $\infty$\=/$\excat{\Gamma}$\=/connected.
    It follows that the composition $\widehat{X} \to X \to C$ is $(n-1)$\=/$\excat{\Gamma}$\=/connected, so $C$ is of type $\FP_n(\excat{\Gamma})$.

    Ad~\localref{prop:2 out of 3:extension}. This is an easy consequence of the Horseshoe Lemma~\cite{buehler}*{Theorem 12.8}.

    Ad~\localref{prop:2 out of 3:kernel}. For $n=0$ this is a consequence of Schanuel's Lemma. For $n \geq 1$ let $P$ be a finitely generated
    projective $\coeffZ[\Gamma]$\=/module with an admissible epimorphism $p \colon P \onto B$. Let $K \defq \ker(p)$ and $K' \defq \ker(\pi \circ p)$.
    By~\cref{prop: schanuel exact categories}, $K$ is of type $\FP_{n-1}(\excat{\Gamma})$ and $K'$ of type $\FP_n(\excat{\Gamma})$. Consider the following
    diagram.
    \[\begin{tikzcd}
        & 0 & K & {K'} \\
        \phantom{0} & 0 & P & P & 0 \\[1ex]
        0 & A & B & C & 0 \\
        & |[alias=D]|A & 0
        \arrow[from=1-2, to=1-3]
        \arrow[hook, from=1-3, to=1-4]
        \arrow[hook', from=1-2, to=2-2]
        \arrow[hook', from=1-3, to=2-3]
        \arrow[hook', from=1-4, to=2-4]
        \arrow[from=2-2, to=2-3]
        \arrow[from=2-3, to=2-4]
        \arrow[phantom, from=2-3, to=3-3, ""{coordinate, name=B}]
        \arrow[phantom, from=2-4, to=2-5, ""{coordinate, name=A}]
        \arrow[phantom, from=3-1, to=3-2, ""{coordinate, name=C}]
        \arrow[hook, from=3-2, to=3-3]
        \arrow[from=3-2, to=4-2]
        \arrow[two heads, from=3-3, to=3-4]
        \arrow[from=3-3, to=4-3]
        \arrow[from=3-4, to=3-5]
        \arrow[from=4-2, to=4-3]
        \arrow[from=1-4, to=4-2, to path={
                ([xshift=-0.5ex]\tikztostart.east) -- (\tikztostart.east)
                to[out=0, in=90, looseness=1] ($ (A) + (0.5ex,3ex) $)
                to[out=-90, in=0, looseness=1.1] ($ (B) + (0ex,0.5ex) $)
                to[out=180, in=90, looseness=1.1] ($ (C) - (0.5ex,2.5ex) $)
                to[out=-90, in=180, looseness=1] ([xshift=-0.5ex]D.west) -- (D.west)
            }]
        \arrow[from=2-2, to=3-2, crossing over]
        \arrow[two heads, from=2-3, to=3-3, crossing over]
        \arrow[two heads, from=2-4, to=3-4, crossing over]
        \arrow[from=2-4, to=2-5, crossing over]
        \arrow[from=3-1, to=3-2, crossing over]
    \end{tikzcd}\]
    The Snake Lemma~\cite{buehler}*{Proposition 8.11} yields the connecting homomorphism $K' \to A$ such that
    $0 \to K \to K' \to A \to 0$ is an admissible short exact sequence. It follows from \localref{prop:2 out of 3:quotient}
    that $A$ is of type $\FP_n(\excat{\Gamma})$.
\end{proof}

\begin{thm}\label{thm: blow-up FPnE}
    Let $n\in\bbN \cup \{\infty\}$ and let $M$ be a $\coeffZ[\Gamma]$\=/module with a $\coeffZ[\Gamma]$\=/resolution $\varepsilon \colon C \onto M$.
    Assume that $\varepsilon$ is $(n - 1)$\=/$\excat{\Gamma}$\=/connected and for $j \geq 0$ the module $C_j$ is of type $\FP_{n-j}(\excat{\Gamma})$.
    Then $M$ is of type $\FP_n(\excat{\Gamma})$.
\end{thm}
\begin{proof}
    This can be readily proved using~\cref{thm: blow-up}. However, we give here an alternative proof that uses~\cref{prop:2 out of 3}.
    The resolution $C \onto M$ splits into a sequence of admissible short exact sequences
    \[\begin{tikzcd}
        {C_{n}} &[-4ex] &[-4ex] {C_{n-1}} &[-4ex] &[-4ex] {C_{n-2}} & \cdots &[-4ex] &[-4ex] {C_0} &[2ex] M, \\[-1ex]
        & {K_{n-1}} && {K_{n-2}} &&& {K_0}
        \arrow["{\partial_{n}}", from=1-1, to=1-3]
        \arrow[two heads, from=1-1, to=2-2]
        \arrow["{\partial_{n-1}}", from=1-3, to=1-5]
        \arrow[two heads, from=1-3, to=2-4]
        \arrow["{\partial_{n-2}}", from=1-5, to=1-6]
        \arrow["{\partial_1}", from=1-6, to=1-8]
        \arrow[two heads, from=1-6, to=2-7]
        \arrow["{\partial_0 = \varepsilon}", two heads, from=1-8, to=1-9]
        \arrow[hook, from=2-2, to=1-3]
        \arrow[hook, from=2-4, to=1-5]
        \arrow[hook, from=2-7, to=1-8]
    \end{tikzcd}\]
    where $K_i \defq \ker(\partial_i)$. Since $C_{n}$ is of type $\FP_0(\excat{\Gamma})$, so is $K_{n-1}$.
    Considering the admissible exact sequences $0 \to K_i \to C_i \to K_{i-1} \to 0$, we inductively obtain
    that $K_i$ is of type $\FP_{n-i-1}(\excat{\Gamma})$ through application of~\cref{prop:2 out of 3:quotient}.
    Since $K_0$ is of type $\FP_{n-1}(\excat{\Gamma})$ and $C_0$ is of type $\FP_n(\excat{\Gamma})$ a final application
    of~\cref{prop:2 out of 3:quotient} yields that $M$ is of type $\FP_n(\excat{\Gamma})$.
\end{proof}

\section{Resolutions and polynomial filling functions}\label{sec: resolutions and fillings}
This section explains the connection between filling functions (see~\cref{sec: filling functions}) and exactness in the exact category
$\excat{\Gamma}$. It identifies connectivity of a chain complex $X$
with respect to $\excat{\Gamma}$ with the existence of a (partial) $\sfB_\Gamma$\=/bounded
non-equivariant chain contraction. The key point is that the Cayley bornology is chosen so that
boundedness directly corresponds to polynomial bounds for a weighted version of the homological
filling functions.

\begin{prop}\label{prop: connected iff bounded contraction}
    Let $X$ be a $\coeffZ[\Gamma]$\=/chain complex.
    Then $X$ is $n$\=/$\excat{\Gamma}$\=/connected if and only if there exists a partial $\coeffZ$\=/linear homotopy $h \colon X_k \to X_{k+1}$ such that $h_k$ is $\sfB_\Gamma$\=/bounded and
    $\partial h_k + h_{k-1} \partial = \id_{X_k}$ for $k \leq n$.
\end{prop}
\begin{proof}
    Throughout let $K_k \defq \ker(\partial_k)$.
    Assume that $X$ is $n$\=/$\excat{\Gamma}$\=/connected.
    Then $\partial_k$ is admissible for $k \leq n + 1$, hence factors as $X_k \to K_{k-1} \to X_{k-1}$,
    where $X_k \onto K_{k-1}$ is an admissible epimorphism. In particular, the sequence $K_k \injectarr X_k \onto K_{k-1}$
    is $\excat{\Gamma}$\=/exact, that is, it splits
    along $\sfB_\Gamma$\=/bounded $\coeffZ$\=/linear maps $r_k \colon X_k \to K_k$
    and $s_{k}\colon K_{k-1} \to X_k$.
    By~\cref{lem: mono with bounded retraction is admissible}, we may assume that $r_k \circ s_k = 0$
    and $\iota_k \circ r_k + s_k \circ \partial_k = \id_{X_k}$, where $\iota_k$ denotes the inclusion $K_k \injectarr X_k$.
    For $k \leq n$ we obtain a decomposition
    \[\begin{tikzcd}
        {X_{k+1}} & {K_k} & {X_k}.
        \arrow["{\partial_{k+1}}", two heads,, from=1-1, to=1-2]
        \arrow["{s_{k+1}}", bend left=40, shift right=1, from=1-2, to=1-1]
        \arrow["{\iota_k}", hook, from=1-2, to=1-3]
        \arrow["{r_{k}}", bend left=40, shift right=1, from=1-3, to=1-2]
    \end{tikzcd}\]
    Define $h_k = s_{k+1} \circ r_k$ for $k \leq n$ and $h_k = 0$ otherwise. Then $h_k$ is $\sfB_\Gamma$\=/bounded
    and satisfies $\partial h_k + h_{k-1} \partial = \id$ for $k \leq n$.

    For the converse direction, it is straightforward to check that $X$ is $\coeffZ$\=/exact up to degree $n$. It therefore remains
    to show that the differentials of $X$ are admissible morphisms up to degree $n + 1$.
    Let $s_k \colon X_k \to K_k$ be given by $s_k = \partial \circ h_k$. Because $\partial$ is $\Gamma$-equivariant, it is $\sfB_\Gamma$\=/bounded by~\cref{lem: every equivariant map is bounded}, hence $s_k$ is $\sfB_\Gamma$\=/bounded as the
    composition of $\sfB_\Gamma$\=/bounded maps.
    It follows that $K_k \injectarr X_k$ is admissible.
    In particular, \cref{lem: induced bornology and cayley bornology equal for admissible map} implies that the subspace bornology and the
    Cayley bornology on $\coeffQ \otimes_\coeffZ K_k$ coincide.
    A section of $X_{k+1} \onto K_k$ is given by $h_k \circ \iota_k$ which, by the preceding observation, is $\sfB_\Gamma$\=/bounded.
    Therefore, by~\cref{lem: mono with bounded retraction is admissible} $X_{k+1} \onto K_k$ is admissible, and we conclude that
    $X_\ast$ is $n$\=/$\excat{\Gamma}$\=/connected.
\end{proof}

\begin{prop}\label{prop: partial inverse iff connected}
    Let $f \colon X \to Y$ be a map of $\coeffZ[\Gamma]$\=/chain complexes.
    Assume $f$ is $n$\=/$\excat{\Gamma}$\=/connected. Then there is a partial $\coeffZ$\=/linear chain map $g \colon Y \to X$ and partial $\coeffZ$\=/linear homotopies
    $h^X_k \colon X_k \to X_{k+1}$, $h^Y_k \colon Y_k \to Y_{k+1}$ such that
    \begin{enumerate}
        \item $\partial g_k = g_{k-1}\partial$ for $k \leq n$;
        \item $\partial h^X_k + h^X_{k-1}\partial = g_k \circ f_k - \id$ for $k \leq n - 1$;
        \item $\partial h^Y_k + h^Y_{k-1}\partial = f_k \circ g_k - \id$ for $k \leq n$;
        \item $h^X_{k-1}$, $h^Y_k$, $g_k$ are $\sfB_\Gamma$\=/bounded.
    \end{enumerate}
\end{prop}
\begin{proof}
    Assume that $f$ is $n$\=/$\excat{\Gamma}$\=/connected. Let $h_k \colon \cone(f)_k \to \cone(f)_{k+1}$ be the homotopy obtained from~\cref{prop: connected iff bounded contraction}. The map $h_k$ can be written as
    \[
        h_k(x,y) = \bigl(h_k^{11}(x) + h_k^{12}(y), h_k^{21}(x) + h_k^{22}(y)\bigr).
    \]
    Clearly all $h_k^{ij}$ are $\sfB_\Gamma$\=/bounded for $k \leq n$.
    One easily checks that the homotopy condition $\partial h_k + h_{k-1}\partial = \id$ implies that $g_k \coloneqq h_k^{12}$
    is a partial chain map and $h^X_k \coloneqq h_k^{11}$ and $h^Y_k \coloneqq -h_k^{22}$ are the desired partial homotopies for $k \leq n$. For $k > n$ all maps can simply be chosen to be trivial.
\end{proof}

\subsection{Bounded resolutions from polynomial filling functions}\label{sec: filling functions}
Let $f \colon M \to N$ be a $\coeffZ[\Gamma]$\=/linear map between $\coeffZ[\Gamma]$\=/permutation modules.
We consider the $\ell^1$-norm $\norm{\placeholder}=\norm{\placeholder}^\Gamma_0$ and the weighted norm $\norm{\placeholder}^\Gamma_1$ from
\cref{exa: bornology permutation modules} on $M$ and $N$.

The \emph{filling volume} of an element $c \in \im(f)$ is defined as
\[
    \filling_{\im(f)}(c) = \inf_{\substack{b\in M\\ f(b)=c}} \norm{b}.
\]
The \emph{weighted filling volume} of an element $c \in \im(f)$ is defined as
\[
    \weightedfilling_{\im(f)}(c) = \inf_{\substack{b\in M\\ f(b)=c}}\norm{b}_1^\Gamma.
\]
The \emph{filling function} of $f$ is defined as
\[
    \filling_f(v) = \;\sup_{\mathclap{\substack{c\in \im(f)\\\norm{c} \le v}}}\; \filling_{\im(f)}(c).
\]
The \emph{weighted filling function} of $f$ is defined as
\[
    \weightedfilling_f(v) = \;\sup_{\mathclap{\substack{c\in \im(f)\\\norm{c}^\Gamma_1\le v}}}\; \weightedfilling_{\im(f)}(c).
\]
For an $(n-1)$-acyclic $\coeffZ[\Gamma]$-chain complex $X$ consisting of permutation modules we write $\filling^n_X$ for $\filling_{\partial^X_n}$
and $\weightedfilling^n_X$ for $\weightedfilling_{\partial^X_n}$.
In this case $\im \partial^X_n = \ker \partial^X_{n-1}$, hence
$\filling^n_X$ and $\weightedfilling^n_X$ measure the difficulty of filling $(n-1)$-cycles by $n$-chains in $X$.

\begin{defn}\label{def: geometric filling function}
    Let $X$ be a $\Gamma$\=/simplicial complex. Then the cellular chain complex $C_\ast(X;\coeffZ)$
    consists of $\Gamma$\=/permutation modules and we define
    $\filling^n_{\coeffZ,X} \defq \filling^n_{C_\ast(X;\coeffZ)}$
    and $\weightedfilling^n_{\coeffZ,X} \defq \weightedfilling^n_{C_\ast(X;\coeffZ)}$.
\end{defn}

\begin{defn}\label{def: general def of homological Dehn function}
    Let $n \geq 0$. A group $\Gamma$ is said to be of type $\FP_n(\coeffZ)$ if there exists a projective resolution $P_\ast \onto \coeffZ$ of the trivial $\coeffZ[\Gamma]$\=/module $\coeffZ$ such that $P_i$ is finitely generated as a $\coeffZ[\Gamma]$\=/module for $i \leq n$.
    By~\cite{brown}*{VIII 4.3}, this is equivalent to the existence of a free $\coeffZ[\Gamma]$\=/resolution with the same finiteness properties.

    Let $\Gamma$ be of type $\FP_n(\coeffZ)$ and $L_\ast \onto \coeffZ$ be such a resolution.
    The \Nth{n} \emph{homological Dehn function} of $\Gamma$ is $\filling^n_{\coeffZ,\Gamma} \coloneqq \filling^n_{L_\ast}$
    and the \Nth{n} \emph{weighted homological Dehn function} of $\Gamma$ is $\weightedfilling^n_{\coeffZ,\Gamma} \coloneqq \weightedfilling^n_{L_\ast}$.
\end{defn}

While these definitions depend on the choice of resolution, the filling functions are well-defined up to linear equivalence (\cref{def:linear-equivalence}).
Their equivalence class is also invariant under equivalence of norms~\cite{weis}*{Lemma 2.14}.
Therefore, one could also extend \cref{def: coefficients} to normed rings which are $\varepsilon$-separated for $0 < \varepsilon \leq 1$.
Moreover, both the weighted and unweighted homological filling functions are quasi-isometry invariants for groups of type $\FP_n(\coeffZ)$ under the assumption that they take finite values~\cite{weis}*{Theorems 5.7, 5.10}, which is the case for $\coeffZ = \bbZ$ or if $\coeffZ$ is
equipped with the discrete norm~\cites{fleming+martinez-pedroza,weis}.

For a Riemannian manifold, the classical homological filling functions are defined using integral Lipschitz singular chains:
the size of a chain is its Riemannian mass, and one minimizes the mass of Lipschitz chains filling a given Lipschitz cycle~\cite{gromov-filling}.
In the situations where both viewpoints apply, these filling functions agree with the cellular ones up to $\approx$\=/equivalence (see~\cref{def:linear-equivalence}).
This is a consequence of the Federer-Fleming deformation technique~\citelist{\cite{abrams_etal}*{Section~2}\cite{epstein}*{10.3}}.

\begin{defn}\label{def:linear-equivalence}
    Given functions $f, g \colon \bbR_{\ge 0} \to \bbR_{\ge 0} \cup \{\infty\}$ we say that $f \preccurlyeq g$ if there are constants $C, D > 0$ such that $f(v) \leq C g(Cv + D) + Cv + D$ for every $v \geq 0$. We say $f \approx g$ if both $f \preccurlyeq g$ and $g \preccurlyeq f$ hold.
\end{defn}

\begin{lem}[\citelist{\cite{bader+kropholler+vankov}*{Lemma 2.12}\cite{weis}*{Lemma 2.12}}]
    Up to $\approx$-equivalence the definitions of $\filling^n_{\coeffZ,\Gamma}$ and $\weightedfilling^n_{\coeffZ,\Gamma}$ do not depend on the choice of free resolution.
\end{lem}

\begin{defn}
    We say that a function $f \colon \bbR_{\ge 0} \to \bbR_{\ge 0}$ is \emph{polynomial} if there is some $k\in\bbN$ and $C>0$ such that $f(v)\le C\cdot v^k$ for every $v \geq 1$. Equivalently $f$ is polynomial if $f \preccurlyeq p$ for some polynomial $p$.
\end{defn}

The following \lcnamecref{prop: group polynomial equivalence} allows us to work with
$\filling^n_{\coeffZ,\Gamma}$ instead of $\weightedfilling^n_{\coeffZ,\Gamma}$. Note that while $\filling^1_{\coeffZ,\Gamma}$ is only finite
for non-infinite groups, $\weightedfilling^1_{\coeffZ,\Gamma}$ is always quadratic (see~\cite{weis}*{Remark 4.7})
and both $\filling^0_{\coeffZ,\Gamma}$ and $\weightedfilling^0_{\coeffZ,\Gamma}$ are always linear.

\begin{prop}[\cite{weis}*{Proposition 4.6}]\label{prop: group polynomial equivalence}
    Let $\Gamma$ be a group of type $\FP_n(\coeffZ)$ where $n \geq 2$. Then $\filling^n_{\coeffZ,\Gamma}$ is polynomial if and
    only if $\weightedfilling^n_{\coeffZ,\Gamma}$ is polynomial.
\end{prop}

\begin{prop}\label{prop: bounded contraction from polynomial filling}
    Let $\varepsilon \colon X \onto \coeffZ$ be a $\coeffZ[\Gamma]$\=/resolution such that $X_i$ is a permutation module for $i \leq n$.
    Then $\weightedfilling^k_X$ is polynomial for all $k \leq n$ if and only if there exists a $\coeffZ$\=/linear chain contraction
    $h_k \colon X_k \to X_{k+1}$ such that $h_k$ is $\sfB_\Gamma$\=/bounded for $k \leq n - 1$.
\end{prop}
\begin{proof}
    By~\cref{rem: integral boundedness permutation}, a homomorphism $f \colon X_i \to X_{i+1}$ induces a bounded map $\Qmap{f}$
    if there exist $s\in\bbN$, $C>0$ such that $\norm{f(e)}^\Gamma_1 \leq C \cdot \norm{e}^\Gamma_s$ for every~$\coeffZ$-basis element $e$.

    We construct $h_\ast$ inductively. We define $h_{-1}\colon\coeffZ\to X_0$ to be any choice of section.
    Suppose there are homomorphisms $h_i\colon X_i\to X_{i+1}$ for $-1\le i\le k-1$ such that $h_i$ is $\sfB_\Gamma$\=/bounded and
    \[\partial h_i+h_{i-1}\partial=\id\]
    for every $-1\le i\le k-1$.
    If $0 \leq k<n$, for every $\coeffZ$-basis element $e \in X_k$ we choose an \emph{optimal}
    boundary $b\in X_{k+1}$, that is,
    \[\norm{b}_1^\Gamma \leq \weightedfilling_{\partial_{k+1}}(e') + 1,\]
    of the $k$-cycle $e'=e-h_{k-1}(\partial e)$.
    Since $\weightedfilling_X^i$ is polynomial for $1\le i\le n$ and $k+1\le n$,
    there is an exponent $s\in \bbN$ and $C>0$ such that
    \[ \norm{b}^\Gamma_1 \leq C \cdot \bigl(\norm{e'}^\Gamma_1\bigr)^s + 1\]
    for every~$\coeffZ$-basis element $e$.
    Because $\coeffZ$ is $1$\=/separated, $\norm{e'}^\Gamma_1 \geq 1$, so after possibly enlarging $C$, we can assume
    $\norm{b}^\Gamma_1 \leq C \cdot \bigl(\norm{e'}^\Gamma_1\bigr)^s$.

    Setting $h_k(e)=b$ on a $\coeffZ$-basis, we obtain a homomorphism $X_k \to X_{k+1}$. For $k \geq n$ we let $h_k$ be the trivial map.

    Let $k-1<n$. By $\sfB_\Gamma$\=/boundedness of $h_{k-1}$, there is an exponent $t\in\bbN$ and $\widetilde C>0$ such that $\norm{e'}^\Gamma_1 \leq \widetilde C\cdot \norm{e}^\Gamma_t$ for every~$\coeffZ$-basis element $e$.
    Hence, we obtain that
    \[ \norm{h_k(e)}^\Gamma_1=\norm{b}^\Gamma_1\leq C\cdot \bigl(\norm{e'}^\Gamma_1\bigr)^s\leq C \widetilde C^s\cdot \bigl(\norm{e}^\Gamma_t\bigr)^s=C \widetilde C^s\cdot \norm{e}^\Gamma_{ts}\]
    for every $\coeffZ$-basis element $e$. This means that $h_k$ is $\sfB_\Gamma$\=/bounded.

    For the converse, it is clear that $\weightedfilling_X^{0}(v)$ is linear.
    Let $1\le k\le n$ and let $c\in K_{k-1}\coloneqq\ker(\partial^X_{k-1})$ be a cycle.
    Since $\partial c=0$, we have $\partial_k h_{k-1}(c)=c$.
    Thus, $h_{k-1}(c)$ is a filling of $c$.
    Since $h_{k-1}$ is $\sfB_\Gamma$-bounded and $X_{k-1}$, $X_k$ are permutation modules,
    \cref{rem:characterisation-bounded} gives constants $C>0$ and $r\in\bbN$ such that
    \[
        \norm{h_{k-1}(x)}^\Gamma_1 \le C\norm{x}^\Gamma_r.
    \]
    By $1$-separation of $\coeffZ$, we have $\norm{x}^\Gamma_r \leq (\norm{x}^\Gamma_1)^r$
    and therefore $\weightedfilling_X^{k}(v) \leq C \cdot v^r$.
\end{proof}

\begin{rem}
    Under the weaker assumption that $X_\ast \onto \coeffZ$ is not a resolution but only $(n-1)$\=/$\coeffZ$\=/acyclic, one still obtains a partial
    $\sfB_\Gamma$\=/bounded chain contraction up to degree $n-1$.
\end{rem}

\begin{cor}\label{cor: polynomial filling iff every free resolution bounded}
    Let $\varepsilon \colon X \onto \coeffZ$ be a $\coeffZ[\Gamma]$\=/resolution such that $X_i$ is a permutation module for $i \leq n$.
    Then $X \onto \coeffZ$ is $(n-1)$\=/$\excat{\Gamma}$\=/connected if and only if
    $\weightedfilling^k_X$ is polynomial for all $k \leq n$.
\end{cor}
\begin{proof}
    Apply \cref{prop: connected iff bounded contraction,prop: bounded contraction from polynomial filling}.
\end{proof}

\begin{thm}\label{thm: polynomial fillings from FP_n(E)}
    Let $\Gamma$ be a group of type $\FP_n(\coeffZ)$. Then $\weightedfilling_{\coeffZ,\Gamma}^k$ is polynomial for all $k \leq n$ if and only if
    $\coeffZ$ is of type $\FP_n(\excat{\Gamma})$.
\end{thm}
\begin{proof}
    If $\weightedfilling_{\coeffZ,\Gamma}^k$ is polynomial for all $k \leq n$, it follows from
    \cref{cor: polynomial filling iff every free resolution bounded} that $\coeffZ$ is of type $\FL_n(\excat{\Gamma})$.

    If $\coeffZ$ is of type $\FP_n(\excat{\Gamma})$ it is also of type $\FL_n(\excat{\Gamma})$ by \cref{thm: FP_n equivalent to FL_n}, hence \cref{cor: polynomial filling iff every free resolution bounded}
    implies that $\weightedfilling_{\coeffZ,\Gamma}^k$ is polynomial for all $k \leq n$.
\end{proof}

\subsection{Comparison to geometry and filling functions}
An easy application of~\cref{thm: blow-up FPnE,cor: polynomial filling iff every free resolution bounded} together with~\cref{prop: FPn induction}
is the following theorem, similar in style to a theorem of Brown for finiteness properties~\cite{brown_finiteness}.

\begin{thm}\label{thm: geometric brown theorem}
    Let $\Gamma$ be a finitely generated group.
    Let $X$ be a $(d-1)$\=/acyclic $\Gamma$\=/simplicial complex with cocompact $d$\=/skeleton such that the cellular chain complex $C_\ast(X;\coeffZ) \onto \coeffZ$ is $(d-1)$\=/$\excat{\Gamma}$\=/connected.
    Assume for each stabilizer $\Gamma_\sigma$ of a $k$-cell $\sigma$ that
    \begin{condenum}
        \item $\Gamma_\sigma$ is of type $\FP_{d-k}(\coeffZ)$;
        \item $\filling_{\coeffZ,\Gamma_\sigma}^i$ is polynomial for $2 \le i \leq d - k$;
        \item $\Gamma_\sigma$ is finitely generated and at most polynomially distorted in $\Gamma$.
    \end{condenum}
    Then $\filling_{\coeffZ,\Gamma}^i$ is polynomial for $2 \le i \leq d$.
\end{thm}

\begin{rem}
    \Cref{thm: geometric brown theorem} deduces polynomiality of $\filling^n_\Gamma$ from
    the polynomiality of filling functions for vertex stabilizers.
    In~\cite{li+sanchez-saldana}*{Proposition 2.4, Remark 2.6} it is shown that one can deduce finiteness properties of stabilizers of $d$-cells from
    those of the stabilizers in the other dimensions together with the finiteness properties of $\Gamma$.
    Utilizing \cref{prop:2 out of 3} one can obtain similar results by the same techniques.
\end{rem}

The difficult part when trying to apply~\cref{thm: geometric brown theorem} lies in checking whether $C_\ast(X;\coeffZ) \onto \coeffZ$ is $(d-1)$\=/$\excat{\Gamma}$\=/connected, that is, whether it admits
a $\sfB_\Gamma$\=/bounded partial chain contraction up to degree $d-1$. We have seen that for groups this is equivalent
to $\filling_{\coeffZ,\Gamma}^k$ being polynomially bounded for $2 \leq k \leq d$, and one might expect the same to happen in the case of spaces.
For the weighted filling functions of $X$, a consequence of
\cref{cor: polynomial filling iff every free resolution bounded} is
the following \lcnamecref{prop: bounded chain contraction from filling}.

\begin{prop}\label{prop: bounded chain contraction from filling}
    Let $X$ be a $(d-1)$\=/acyclic $\Gamma$\=/simplicial complex. Then $\weightedfilling^i_{\coeffZ,X}$ is polynomial for all $i \leq d$ if and only if $C_\ast(X;\coeffZ)\onto \coeffZ$ is $(d-1)$\=/$\excat{\Gamma}$\=/connected.
\end{prop}

Let $\Gamma$ be a finitely generated group. Let $X$ be a connected $\Gamma$\=/simplicial complex.
We endow $X$ with the path metric~$d$ that restricts to the standard Euclidean metric on each simplex. The $n$\=/skeleton $X^{(n)}$ fits into an equivariant pushout diagram
\[\begin{tikzcd}
    \bigsqcup_{i\in I_n} \Gamma/\Gamma_i\times \partial\Delta^n\ar[d,hook]\ar[r,hook] & X^{(n-1)}\ar[d, hook]\\
    \bigsqcup_{i\in I_n} \Gamma/\Gamma_i\times \Delta^n\ar[r,hook] & X^{(n)},
\end{tikzcd}\]
where $I_n$ is a set of orbit representatives of $n$-cells. For each dimension~$n$ we pick a base simplex among the $n$-simplices of $X$.
We declare and compare two weight functions on the set of $n$-simplices.

\begin{defn}
    Let $\Gamma$ be a finitely generated group, and let $X$ be a connected $\Gamma$\=/simplicial
    complex.
    For an $n$-simplex $\sigma$ of $X$ let $d_\sigma$ be the distance of its barycenter to the barycenter of the base $n$-simplex. The \emph{geometric weight function} $\geomweight$ on the set of $n$-simplices is defined by $\geomweight(\sigma)= 1+ d_\sigma$.

    The \emph{algebraic weight function} $\algweight$ on the set of $n$-simplices is defined by
    \[
        \algweight(\{\gamma\Gamma_i\}\times \Delta^n)= 1 + \min \{\ell_\Gamma(\gamma') \mid \gamma' \in \gamma\Gamma_i \}.
    \]
    Associated to each weight function $w \in \{\geomweight,\algweight\}$ we consider the family of norms
    \[
        \norm{\sum_{\sigma} a_\sigma \sigma}^w_k \coloneqq \sum_\sigma \abs{a_\sigma} \cdot w(\sigma)^k.
    \]
    Note that $\norm{\placeholder}^{\algweight}_n$ is the norm described in \cref{exa: bornology permutation modules}.

    For $k = 1$ the norms induce filling functions as in \cref{sec: filling functions}.
    In the case of the algebraic weight the corresponding functions are $\weightedfilling^n_{\coeffZ,X}$. For the geometric weight we denote the resulting filling functions
    by $\geomweightedfilling^n_{\coeffZ,X}$.
\end{defn}

\begin{lem}\label{lem: comparison geometric vs group-theoretic weights on cellular chains}
    Let $X$ be as before. Let $\geomweight, \algweight$ be the geometric and group-theoretic weight functions on the set of $n$-simplices of $X$. Assume that $X$ has cocompact $n$-skeleton.
    \begin{refenum}
        \item\label{lem: comparison weights: algebraic dominates geometric}
              There exists $C>1$ such that $\geomweight\le C\cdot \algweight$.
        \item\label{lem: comparison weights: geometric dominates algebraic iff locally finite}
              There exists $D>0$ such that $\algweight \le \geomweight + D$ if and only if $X$ is locally finite.
    \end{refenum}
\end{lem}
\begin{proof}
    Ad~\localref{lem: comparison weights: algebraic dominates geometric}. Let $K\subset X$ be a compact subset of~$X$ consisting of the finitely many simplices $\{e\Gamma_i\}\times \Delta^n$, $i\in I_n$. Let
    \[
        R=\max\{ d(sx, y)\mid s\in S,x,y\in K\}<\infty.
    \]
    We may assume that $X$ has at least one edge, so $R \ge 1$.
    Then one easily sees that
    \[
        d_{\{\gamma\Gamma_i\}\times \Delta^n} \leq R+R\cdot \ell(\gamma) = R \cdot (1 + \ell(\gamma)),
    \]
    hence $\geomweight \le 1 + R \cdot \algweight \le (R + 1)\cdot \algweight$.

    Ad~\localref{lem: comparison weights: geometric dominates algebraic iff locally finite}.
    Assume that $X$ is locally finite.
    Then there are only finitely many $n$-simplices at
    distance at most~$3R$ from the base $n$-simplex~$\sigma_0$.
    Therefore, there is a finite subset $S\subset \Gamma$ such that if the distance of an $n$-simplex $\{\gamma\Gamma_i\}\times \Delta^n$ to the base $n$-simplex~$\sigma_0$ is at most~$3R$, then $\gamma=sh$ for some $s\in S$ and $h\in \Gamma_i$.
    By enlarging~$S$, we may assume that $S$ generates~$\Gamma$.
    We may also assume that the word metric and thus $\algweight$ are defined with respect to~$S$.

    Next, consider an $n$-simplex $\sigma$. Let $\alpha$ be a path (between barycenters) that approximates the distance of~$\sigma$ to $\sigma_0$ within~$R$, that is, $\ell(\alpha) \leq d_\sigma + R$. Let $m\in\bbN$ be such that $(m-1)R< \ell(\alpha)\le mR$.
    Then
    \[
        m \leq \frac{\ell(\alpha)}{R} + 1 \leq \frac{d_\sigma + R}{R} + 1 \leq d_\sigma + (R + 1).
    \]
    We divide the path~$\alpha$ by $(m+1)$ equally spaced points $x_0,\dots, x_m$, where $x_0$ is the barycenter of $\sigma_0$ and $x_m$ is the barycenter of $\sigma_m\defq\sigma$.
    For every other point $x_k$, $k\in\{1,\dots, m-1\}$, we pick an $n$-simplex $\sigma_k$ that is at most distance $R$ from~$x_k$. Set $\delta_0=e$.
    Since the distance of $\sigma_{k-1}$ and $\sigma_k$ is at most~$3R$, we can choose $s_k\in S$ inductively such that $\delta_k\defq \delta_{k-1}s_k$ represents $\sigma_k$.
    Hence, $\delta_m=s_1\cdots s_m$ represents $\sigma$, and therefore
    $\algweight(\sigma)\le 1+m \le \geomweight(\sigma) + (R + 1)$.

    Conversely, assume that $\algweight \leq \geomweight + D$ for some $D > 0$ and let $\sigma$ be an $n$-simplex.
    If $\tau$ is another $n$-simplex such that $d(\sigma, \tau) \leq 1$ we have $d_\tau \leq d_\sigma + d(\sigma,\tau) \leq d_\sigma + 1$ by the triangle inequality.
    Therefore, the set of $n$-simplices contained in the simplicial ball of radius $1$ around $\sigma$ is a subset of
    \begin{align*}
        \{ \tau \mid \geomweight(\tau) \leq \geomweight(\sigma) + 1\}
         &\subseteq \{\tau \mid \algweight(\tau) \leq \geomweight(\sigma) + D + 1\} \\
         &= \bigl\{\{\gamma\Gamma_i\} \times \Delta^n \mid \ell(\gamma) \leq \geomweight(\sigma) + D\bigr\}.
    \end{align*}
    The latter set has size at most $\size{I_n} \cdot \size{\{\gamma \mid \ell(\gamma) \leq \geomweight(\sigma) + D\}}$, which is finite as $\Gamma$ is finitely generated and acts cocompactly on $X^{(n)}$.
\end{proof}

\begin{cor}\label{cor: cellular chain complex exact}
    Let $X$ be a $(d-1)$\=/$\coeffZ$\=/acyclic $\Gamma$\=/simplicial complex with cocompact $d$\=/skeleton. If $X$ is locally finite, then $\geomweightedfilling^d_{\coeffZ,X}$ is equivalent to $\weightedfilling^d_{\coeffZ,X}$.
\end{cor}

The following \lcnamecref{lem: polynomial equivalence} is the geometric
analog of~\cref{prop: group polynomial equivalence}.
It is proved for free $\Gamma$\=/complexes in~\citelist{\cite{ji+ramsey}*{Corollary~2.5}\cite{bader+sauer}*{Lemma~6.7}}.
Since the proofs do not use freeness, the result also holds for $\Gamma$\=/complexes with non-trivial stabilizers.

\begin{lem}[\citelist{\cite{ji+ramsey}*{Corollary~2.5}\cite{bader+sauer}*{Lemma~6.7}}]\label{lem: polynomial equivalence}
    Let $X$ be a $(d-1)$\=/$\coeffZ$\=/acyclic simplicial $\Gamma$\=/complex with cocompact $d$\=/skeleton. For $2 \le i\le d$, the function $\filling_{\coeffZ,X}^i$ is polynomial if and only if $\geomweightedfilling^i_{\coeffZ,X}$ is polynomial.
\end{lem}

\begin{cor}\label{cor: exactness from geometric filling}
    Let $X$ be a locally finite $(d-1)$\=/$\coeffZ$\=/acyclic $\Gamma$\=/simplicial complex with cocompact $d$\=/skeleton. Then $\filling^i_{\coeffZ,X}$ is polynomial for all $2 \leq i \leq d$ if and only if $C_\ast(X;\coeffZ) \onto \coeffZ$ is $(d-1)$\=/$\excat{\Gamma}$\=/connected.
\end{cor}

By a \CAT{0}\=/\emph{simplicial complex} we mean a $M_0$-simplicial complex in the sense of Bridson-Haefliger~\cite{bridson+haefliger}*{Chapter~I.7}.

\begin{cor}\label{cor: CAT(0)}
    Let $X$ be a locally finite \CAT{0}\=/simplicial complex that admits a cocompact simplicial group action by a finitely generated group $\Gamma$. Then $C_\ast(X;\bbZ)$ is an $\excat{\Gamma}$\=/resolution of~$\bbZ$.
\end{cor}

\begin{proof}
    By work of Gromov~\cite{gromov-filling}*{2.3} and Wenger~\cite{wenger}, the higher filling functions of a \CAT{0}-space, defined in terms of singular Lipschitz cycles and Hausdorff volume, are polynomial in all degrees $i \geq 2$. By~\citelist{\cite{abrams_etal}*{Section~2}\cite{epstein}*{10.3}}, these geometric filling functions coincide with their combinatorial counterparts $\filling_{\bbZ,X}^i$ under the assumptions on~$X$. Therefore, $\filling_{\bbZ,X}^i$ is polynomial for every $i\ge 2$.
    \Cref{cor: exactness from geometric filling} now implies the statement.
\end{proof}

We do not know whether \cref{cor: exactness from geometric filling} can be extended to complexes which are not locally finite. In~\cite{webb} Webb gives examples of
locally infinite complexes $X$ which are hyperbolic and contractible but do not admit any combinatorial
isoperimetric inequality. In our notation this means that $\filling^2_X$ does not
take finite values. While this does not give an immediate counterexample
to a version of \cref{cor: exactness from geometric filling} for locally infinite
complexes, it highlights the pathologies one needs to take into consideration.
We therefore ask the following question, to which we expect a negative answer in general.
A positive answer for $1$\=/dimensional complexes is given in \cref{sec: graphs of groups}.

\begin{question}\label{question: polynomial fillings imply connected chain complex}
    Let $X$ be a non-locally finite $(d-1)$\=/$\coeffZ$\=/acyclic $\Gamma$\=/simplicial complex with cocompact $d$\=/skeleton.
    Does polynomiality of $\filling_{\coeffZ,X}^i$ for $2 \leq i \leq d$ imply that $C_\ast(X;\coeffZ) \onto \coeffZ$ is $(d-1)$\=/$\excat{\Gamma}$\=/connected?
\end{question}

In view of~\cref{lem: polynomial equivalence} it would suffice to show that polynomiality of $\geomweightedfilling_{\coeffZ,X}^i$ and $\weightedfilling_{\coeffZ,X}^i$ are equivalent.

\begin{question}
    Let $X$ be a non-locally finite $(d-1)$\=/$\coeffZ$\=/acyclic $\Gamma$\=/simplicial complex with cocompact $d$\=/skeleton.
    Is $\weightedfilling_{\coeffZ,X}^i$ polynomial if and only if $\geomweightedfilling_{\coeffZ,X}^i$ is polynomial?
\end{question}

\section{Application to group extensions}
In the following we consider extensions $N \hookrightarrow G \onto Q$ of finitely generated groups.
We say an extension is \emph{polynomial} if $N$ is at most polynomially distorted in $G$.

\begin{thm}\label{thm: extension outer to inner}
    Let $N \hookrightarrow G \onto Q$ be a polynomial extension and $n \geq 0$. Assume
    \begin{condenum}
        \item $N$ is of type $\FP_n(\coeffZ)$ and $\weightedfilling^i_{\coeffZ,N}$ is polynomial for all $i \leq n$;
        \item $Q$ is of type $\FP_n(\coeffZ)$ and $\weightedfilling^i_{\coeffZ,Q}$ is polynomial for all $i \leq n$.
    \end{condenum}
    Then $G$ is of type $\FP_n(\coeffZ)$ and $\weightedfilling^i_{\coeffZ,G}$ is polynomial for all $i \leq n$.
\end{thm}
\begin{proof}
    By~\cref{thm: polynomial fillings from FP_n(E)}, $\coeffZ$ is both of type $\FP_n(\excat{Q})$ and $\FP_n(\excat{N})$.
    Then there exists a free $\coeffZ[Q]$\=/resolution $\varepsilon \colon L \onto \coeffZ$ such that $L_k$ is finitely generated
    for $k \leq n$ and $\varepsilon$ is $(n - 1)$\=/$\excat{Q}$\=/connected.
    Let $C = \restriction^Q_G(L)$. By~\cref{prop: bornological restriction exact}, $C$ is $(n-1)$\=/$\excat{G}$\=/connected.
    As $\coeffZ[G]$\=/modules $C_k = \bigoplus_{I_k} \coeffZ[G/N]$.
    For $k \leq n$, $L_k$ is finitely generated, so the index set $I_k$ is finite.
    It follows from~\cref{prop: FPn induction} that $\induction_N^G(\coeffZ) = \coeffZ[G/N]$ is of type $\FP_n(\excat{G})$.
    In particular, $C_k$ is of type $\FP_n(\excat{G})$ if $k \leq n$. Applying~\cref{thm: blow-up FPnE} completes the proof.
\end{proof}

\begin{thm}\label{thm: extension left to right}
    Let $N \hookrightarrow G \onto Q$ be a polynomial extension and $n \geq 0$. Assume
    \begin{condenum}
        \item $N$ is of type $\FP_{n-1}(\coeffZ)$ and $\weightedfilling^i_{\coeffZ,N}$ is polynomial for all $i \leq n-1$;
        \item $G$ is of type $\FP_n(\coeffZ)$ and $\weightedfilling^i_{\coeffZ,G}$ is polynomial for all $i \leq n$.
    \end{condenum}
    Then $Q$ is of type $\FP_n(\coeffZ)$ and $\weightedfilling^i_{\coeffZ,Q}$ is polynomial for all $i \leq n$.
\end{thm}
\begin{proof}
    We proceed inductively. For $n \leq 1$ the statement is trivially true, as $\weightedfilling^0_{\coeffZ,G}$ is always linear and $\weightedfilling^1_{\coeffZ,G}$
    is always bounded by a quadratic function.
    Now assume that the conclusion holds for some $k < n$.
    Then there exists a free $\coeffZ[Q]$\=/resolution $\varepsilon_Q \colon L \onto \coeffZ$, such that
    $L_i$ is finitely generated for $i \leq k$ and $\varepsilon_Q$ is $(k-1)$\=/$\excat{Q}$\=/connected.
    We consider the $\coeffZ[G]$\=/chain complex $C = \restriction^Q_G(L^{(k)})$.
    As in the proof of~\cref{thm: extension outer to inner}, $C_i$ is of type $\FP_{n-1}(\excat{G})$, hence $\FL_{n-1}(\excat{G})$, for $i \leq k$.
    Applying~\cref{thm: blow-up} yields a $\coeffZ[G]$\=/chain complex $\widehat{C}$ together with
    a weak equivalence $q \colon \widehat{C} \to C$ that is $k$\=/$\excat{G}$\=/connected.
    By construction, the complex $\widehat{C}$ consists of finitely generated free $\coeffZ[G]$\=/modules.
    From~\cref{cor: polynomial filling iff every free resolution bounded} it follows that $\widehat{C} \onto \coeffZ$ is $(k-1)$\=/$\excat{G}$\=/connected, due to the assumptions on
    $\weightedfilling_{\coeffZ,G}^i$.
    Thus, the partial resolution
    \[
        \widehat{C}_k \to \widehat{C}_{k-1} \to \cdots \to \widehat{C}_0 \onto \coeffZ
    \]
    is $\excat{G}$\=/exact.
    Because $G$ is of type $\FP_{n}(\coeffZ)$ and $k + 1 \leq n$, there exists a finitely generated $\coeffZ[G]$\=/module
    $\widehat{C}'_{k+1} = \coeffZ[G]^m$ and a surjection $\widehat{C}'_{k+1} \onto \widehat{K}_k \defq \ker(\partial^{\widehat{C}}_k)$
    extending the above partial resolution. Denote this extended partial resolution by $\widehat{C}'$.
    By~\cref{cor: polynomial filling iff every free resolution bounded},
    $\widehat{C}'$ is also $\excat{G}$\=/exact, so as a chain complex it is $k$\=/$\excat{G}$\=/connected.
    Set $L'_{k+1} \defq \coeffZ[Q]^m$ and define $K_k \defq \ker(\partial^L_k)$.
    Because $q_k$ is $G$-equivariant and $C_k = \restriction^Q_G(L_k)$,
    the composition $\widehat{C}'_{k+1} \onto \widehat{K}_k \onto K_k$ descends to a $\coeffZ[Q]$\=/linear surjection $f \colon L'_{k+1} \onto K_k$.
    Let $\widehat{h}_i \colon \widehat{C}_i \to \widehat{C}_{i+1}$ be the partial chain contraction obtained
    from~\cref{prop: connected iff bounded contraction} applied to the chain complex $\widehat{C}'$.
    Further, let $r \colon C \to \widehat{C}$ and $H_i \colon C_i \to C_{i+1}$, be the chain map and homotopy obtained from
    \cref{prop: partial inverse iff connected} applied to $q$. Consider the following diagram
    \[\begin{tikzcd}
        {\coeffZ[Q]^m} &[-1.75em] {L'_{k+1}} & {K_k} & {L_k} &[-1.75em] {C_k} \\
        {\coeffZ[G]^m} & {\widehat{C}'_{k+1}} & {\widehat{K}_k} & {\widehat{C}'_k} & {\widehat{C}_k}
        \arrow[two heads, from=1-2, to=1-3]
        \arrow[two heads, from=2-3, to=1-3]
        \arrow[hook, from=1-3, to=1-4]
        \arrow[from=1-4, to=1-5, Equal]
        \arrow[from=1-1, to=1-2, Equal]
        \arrow[from=2-1, to=2-2, Equal]
        \arrow[from=2-4, to=2-5, Equal]
        \arrow["r_k", shift left, from=1-5, to=2-5]
        \arrow["f", two heads, from=2-2, to=1-2]
        \arrow[two heads, from=2-2, to=2-3]
        \arrow[hook, from=2-3, to=2-4]
        \arrow["q_k", shift left, from=2-5, to=1-5]
        \arrow["{\widehat{h}_k}", curve={height=-13pt}, from=2-4, to=2-2]
    \end{tikzcd}\]
    For $i \leq k$ define
    $h_i \coloneqq q_{i+1} \circ \widehat{h}_i \circ r_i - H_i$, where $q_{k+1} \coloneqq f$.
    Note that $C_{k+1} = 0$, hence $H_k = 0$ and $h_k \colon C_{k} \to L'_{k+1}$ is well-defined.
    An easy computation shows that $\partial h_i + h_{i-1}\partial = \id$ for $i \leq k$ and
    \cref{prop: bornological restriction exact} implies that $h_i$ is $\sfB_Q$\=/bounded for $i \leq k$.
    The result now follows from \cref{prop: bounded contraction from polynomial filling} for $k + 1$.
\end{proof}

\begin{example}\label{exa: nilpotent groups}
    If $G$ is a finitely generated nilpotent group then $\weightedfilling^i_{\coeffZ,G}$ and hence $\filling^i_{\coeffZ,G}$ are polynomial for all $i \geq 2$.

    Recently, Gabriel Pallier provided a proof of this fact in~\cite{pallier}, which goes back to Gromov.
    Here, we give another proof using \cref{thm: extension outer to inner}, arguing by induction on the nilpotency class.

    This clearly holds for finitely generated abelian groups. Moreover, they all have polynomial first homotopical Dehn function $\delta_G$.
    Now let $G$ be a central extension of finitely generated
    nilpotent groups $N$ and $Q$ of strictly smaller nilpotency class. By induction,
    $\weightedfilling^i_{\coeffZ,N}$, $\weightedfilling^i_{\coeffZ,Q}$ and $\delta_Q$ are all polynomial.
    It is a general result that in the case of central extensions the distortion of $N$ in $G$
    is bounded by the Dehn function of $Q$ (see~\cite{kassabov+riley}*{Corollary 2.2}), hence the
    distortion of $N$ in $G$ is at most polynomial.
    We can therefore apply~\cref{thm: extension outer to inner} to conclude polynomiality of $\weightedfilling^i_{\coeffZ,G}$.
    By~\cite{conner}, $\delta_G$ is also polynomial.
\end{example}

\section{Application to graphs of groups}\label{sec: graphs of groups}

In this section we study higher filling functions of the fundamental group $\pi_1(\calG)$ of a graph of groups under suitable
constraints on the edge and vertex groups.
\Cref{cor: cellular chain complex exact} lets us apply~\cref{thm: geometric brown theorem} to the Bass-Serre tree $X$ of
an amalgamated product under the assumption that the tree is locally finite. However, this already implies that
the edge stabilizers have to be of finite index in the vertex stabilizers.
In this section we show that this restriction is not necessary, that is, $C_\ast(X;\coeffZ) \onto \coeffZ$
is always exact with respect to $\excat{\Gamma}$.
This gives a positive answer to \cref{question: polynomial fillings imply connected chain complex} for $1$\=/dimensional complexes.

We achieve this by constructing a $\sfB_\Gamma$\=/bounded chain contraction, which is almost the canonical contraction of
the Bass-Serre tree. We first recall the definition of graphs of groups and their fundamental group.

\begin{defn}\label{def: graph of groups}
    A \emph{graph of groups} $\calG = (Y, \{\Gamma_e\}, \{\Gamma_v\}, \{\iota_e\})$ consists of
    \begin{itemize}
        \item a (not necessarily simplicial) connected finite graph $Y = (V, E)$.
              We consider $Y$ to be unoriented in the sense that if $v$ and $w$ are connected by an
              edge then both $(v,w)$ and $(w,v)$ are contained in $E$. Let $E_{+}$ denote
              a choice of orientation for each edge such that $E = E_{+} \sqcup \overline{E}_{+}$;
        \item for each oriented edge $e \in E$, an \emph{edge-group} $\Gamma_e$, such that $\Gamma_{\overline{e}} = \Gamma_{e}$,
              where $\overline{e}$ is the edge $e$ with the opposite orientation;
        \item for each vertex $v \in V$, a \emph{vertex-group} $\Gamma_v$;
        \item for each edge $e \in E$, a monomorphism $\iota_e \colon \Gamma_e \to \Gamma_{t(e)}$,
              where $t(e)$ denotes the terminal vertex of the edge $e$.
              We also write $i(e)$ for the initial vertex of the edge $e$.
    \end{itemize}
    All vertex- and edge-groups are assumed to be finitely generated.
\end{defn}

A \emph{$\calG$-path} is a sequence $s = (\gamma_0,e_1,\gamma_1,e_2,\ldots,e_k,\gamma_k)$ such that $e_1,\dots,e_k$ is a path in $Y$
and $\gamma_i \in \Gamma_{v_i}$ where $v_0,\ldots,v_k$ are the vertices of this path.
Concatenation of $\calG$-paths with compatible start and end vertices is defined in the obvious way.
A $\calG$-path is called a \emph{$\calG$-loop} if it starts and ends at the same vertex, that is, $v_0 = v_k$.
On the set of $\calG$-paths we consider the equivalence relation generated by the relations
\begin{align*}
    (\gamma,e,\gamma')                     &\sim (\gamma \iota_{\overline{e}}(h),e,\iota_e(h^{-1})\gamma') \\
    (\gamma, e, 1, \overline{e}, \gamma'') &\sim (\gamma\gamma'')
\end{align*}
on subpaths for $e \in E$, $h \in \Gamma_e$, $\gamma,\gamma'' \in \Gamma_{i(e)}$ and $\gamma' \in \Gamma_{t(e)}$.
Given a $\calG$-path we write $[s]$ for the equivalence class of $s$ under this relation.

\begin{defn}
    The \emph{fundamental group $\pi_1(\calG, v_0)$} of the graph of groups $\calG$ is the group of equivalence classes
    of $\calG$-loops based at a fixed vertex $v_0 \in V$.

    The group operation is given by concatenation of loops
    with neutral element the single-vertex path $(1_{\Gamma_{v_0}})$. The inverse of an element $(\gamma_0,e_1,\ldots,e_k,\gamma_k)$ is given by
    $(\gamma_k^{-1},\overline{e}_k,\ldots,\overline{e}_1,\gamma_0^{-1})$.
\end{defn}

An \emph{elementary reduction} of a $\calG$-path is the action of replacing a subpath of the form
$(\gamma,e,\iota_e(h),\overline{e},\gamma')$ by the single-vertex path $(\gamma\iota_{\overline{e}}(h)\gamma')$.
A $\calG$-path is called \emph{reduced} if no further reductions are possible.

\begin{rem}[\cite{chatterji+gautero}*{Remark 2.3}]
    Denote by $\calG_Y$ the free product
    \[
        \calG_Y \defq \bigl(\ast_{v \in V} \Gamma_v \bigr) \ast F_{E_{+}},
    \]
    where $F_{E_{+}}$ is the free group on the set of unoriented edges of $Y$.
    Consider a maximal subtree $T = (V_T, E_T) \subseteq Y$. Then $\pi_1(\calG,v_0) \cong \calG_Y/\llangle R_1 \cup R_2 \rrangle$, where
    \begin{align*}
        R_1 &= \{e = 1, \iota_e(h) = \iota_{\overline{e}}(h) \mid e \in E_T,\, h \in \Gamma_e\} \\
        R_2 &= \{\iota_e(h) = e^{-1}\iota_{\overline{e}}(h)e \mid e \in E \setminus E_T,\, h \in \Gamma_e\}.
    \end{align*}
\end{rem}

Fix a generating set $S_v$ of $\Gamma_v$ for each $v \in V$. Then the union of the $S_v$
together with $E$ is a generating set for $\calG_Y$, hence induces a generating set on $\pi_1(\calG,v_0)$. We denote the resulting length functions
on $\calG_Y$ and $\pi_1(\calG, v_0)$ by $\ell_Y$ and $\ell_\calG$ respectively.
For any $\calG$-path $s = (\gamma_0,e_1,\ldots,e_k,\gamma_k)$ we have
\[
    \ell_Y(s) \leq k + \ell_{\Gamma_{i(e_1)}}(\gamma_0) + \sum_{i=1}^{k} \ell_{\Gamma_{t(e_i)}}(\gamma_i) \eqqcolon L_Y(s),
\]
with equality if $s$ is reduced. Then $\ell_{\Gamma_v}(\gamma) = \ell_Y(\gamma)$ for any $\gamma \in \Gamma_v$ and
\[
    \ell_\calG(\gamma) = \min\{ L_Y(s) \mid [s] = \gamma \}
\]
for all $\gamma \in \pi_1(\calG)$.

\begin{rem}
    Let $v \in V$ be a different base vertex.
    Then a choice of a path from $v_0$ to $v$ induces an isomorphism $f_v \colon \pi_1(\calG,v_0) \to \pi_1(\calG,v)$.
    Denote by $\ell_{\calG_v}$ the length function of $\pi_1(\calG,v)$
    with respect to the base vertex $v$.
    Then $\ell_{\calG_v}$ and $\ell_\calG \circ f_v^{-1}$ are equivalent up to multiplication with a constant.
\end{rem}

\begin{defn}\label{def:bass-serre tree}
    The \emph{Bass-Serre tree of $\calG$} is the tree $\calT = \widetilde{Y} = (\widetilde{V}, \widetilde{E})$
    defined by
    \begin{align*}
        \widetilde{V} = \coprod_{v \in V} \pi_1(\calG)/\Gamma_v,
        \quad
        \widetilde{E} = \coprod_{e \in E} \pi_1(\calG)/\Gamma_e,
    \end{align*}
    that is, the vertices are equivalence classes of $\calG$-paths from $v_0$ to $v$
    and two such vertices $\gamma\Gamma_v$, $\gamma'\Gamma_{v'}$ are connected by an edge
    if $e \defq (v,v')$ is an edge in $Y$ and there exists $h \in \Gamma_{v}$ such that
    $\gamma'\Gamma_{v'} = (\gamma,h,e)\Gamma_{v'}$.
    On $\widetilde{Y}$ we fix $1\Gamma_{v_0}$ as the base vertex.
    Such an edge is represented by the element $(\gamma,h)\Gamma_e$.

    There is a projection $\pi_\calG \colon \widetilde{Y} \to Y$ which maps the vertices $\gamma\Gamma_v$ to $v$ and the edges $\gamma\Gamma_e$ to $e$ for $v \in V$ and $e \in E$.
\end{defn}

\begin{lem}\label{lem: geodesic path in Bass-Serre tree}
    Let $\calG$ be a graph of groups with fundamental group $\Gamma = \pi_1(\calG,v_0)$
    and $\gamma\Gamma_{v_0}$ be a vertex in the corresponding Bass-Serre tree $\widetilde{Y}$,
    lying over the base vertex $v_0$.
    There exists a path $\widetilde{\gamma}$ in $\widetilde{Y}$ from $1\Gamma_{v_0}$ to $\gamma\Gamma_{v_0}$, such that
    \begin{equation}\label{eq: norm bound geodesic Bass-Serre tree}
        \norm{\widetilde{e}_i}^\Gamma_n \leq \norm{\gamma\Gamma_{v_0}}^\Gamma_n
    \end{equation}
    for every edge $\widetilde{e}_i$ in the path $\widetilde{\gamma}$.
\end{lem}
\begin{proof}
    We may assume that $\gamma$ is the $\ell_\calG$-minimizing coset-representative, such that $\norm{\gamma\Gamma_{v_0}}^\Gamma_n = (1 + \ell_\calG(\gamma))^n$.
    Let $\gamma = (\gamma_0,e_1,\ldots,e_k,\gamma_k)$ be a geodesic word for $\gamma$, that is,
    $\ell_\calG(\gamma) = L_Y((\gamma_0,e_1,\ldots,e_k,\gamma_k))$.
    Let $v_0,\ldots,v_k$ be the vertices of the corresponding path in $Y$.
    Note that $(\gamma_0,e_1,\ldots,e_{k-1},\gamma_{k-1})\Gamma_{e_k}$ is an edge in $\widetilde{Y}$
    from $(\gamma_0,e_1,\ldots,\gamma_{k-1})\Gamma_{v_{k-1}}$ to $(\gamma_0,e_1,\ldots,e_k,\gamma_k)\Gamma_{v_k}= \gamma\Gamma_{v_k}$.
    Inductively, we see that
    \begin{align*}
        \widetilde{v}_i     &\defq (\gamma_0,e_1,\ldots,e_i,\gamma_i)\Gamma_{v_i} \in \widetilde{Y} \\
        \widetilde{e}_{i+1} &\defq (\gamma_0,e_1,\ldots,e_{i},\gamma_{i})\Gamma_{e_{i+1}} \in \widetilde{E}
    \end{align*}
    define a path $\widetilde{v}_0,\widetilde{v}_1,\ldots,\widetilde{v}_k$
    in $\widetilde{Y}$ from $1\Gamma_{v_0}$ to $\gamma\Gamma_{v_0}$.
    Moreover,
    \begin{align*}
        \norm{\widetilde{e}_i}^\Gamma_n
         &\leq (1 + L_Y(\gamma_0,e_1,\ldots,e_{i-1},\gamma_{i-1}))^n \\
         &\leq (1 + L_Y(\gamma_0,e_1,\ldots,e_k,\gamma_k))^n \\
         &= (1 + \ell_\calG(\gamma))^n = \norm{\gamma\Gamma_{v_0}}^\Gamma_n
    \end{align*}
    for all $1 \leq i \leq k$.
\end{proof}

\begin{thm}\label{thm: bass serre tree exact}
    Let $\calG$ be a graph of groups with fundamental group $\Gamma = \pi_1(\calG,v_0)$.
    Let $\widetilde{Y}$ be the corresponding Bass-Serre tree.
    Then the augmented cellular chain complex $C_\ast(\widetilde{Y};\coeffZ) \onto \coeffZ$ is $\excat{\Gamma}$\=/exact.
\end{thm}
\begin{proof}
    The cellular chain-complex $C_\ast(\widetilde{Y};\coeffZ)$ of $\widetilde{Y}$ is of the form
    \[
        \cdots \to 0 \longrightarrow \bigoplus_{e \in E_+} \coeffZ[\Gamma/\Gamma_e] \overset{\partial}{\longrightarrow} \bigoplus_{v \in V} \coeffZ[\Gamma/\Gamma_v] \longtwoheadrightarrow \coeffZ.
    \]
    By~\cref{lem: mono with bounded retraction is admissible} it suffices to produce a retraction $s \colon C_0(\widetilde{Y};\coeffZ) \to C_1(\widetilde{Y};\coeffZ)$
    that is $\sfB_\Gamma$\=/bounded.
    For each vertex $v \in V$ pick a path $p_v$ from $v_0$ to $v$ and hence an isomorphism $\pi_1(\calG,v) \cong \pi_1(\calG,v_0)$.

    We first consider vertices $\gamma\Gamma_{v_0}$ lying above the base vertex $v_0$.
    Let $\gamma$ be the $\ell_{\calG}$-minimizing coset-representative in $\gamma\Gamma_{v_0}$
    and let $\widetilde{\gamma} = (\gamma_0,e_1,\gamma_1,\ldots,e_k,\gamma_k)$ be the path in $\widetilde{Y}$ from~\cref{lem: geodesic path in Bass-Serre tree}.
    We define a map $\coeffZ[\Gamma/\Gamma_{v_0}] \to C_1(\widetilde{Y};\coeffZ)$ through
    \[
        s_{v_0}(\gamma\Gamma_{v_0}) \defq \sum_{i=1}^{k} \widetilde{e}_i = \sum_{i=1}^{k} (\gamma_0,e_1,\gamma_1,\ldots,e_{i-1},\gamma_{i-1})\Gamma_{e_i}.
    \]
    Because $\widetilde{\gamma}$ is a path from $1\Gamma_{v_0}$ to $\gamma\Gamma_{v_0}$, we have
    $\partial s_{v_0}(\gamma\Gamma_{v_0}) = \gamma\Gamma_{v_0} - 1\Gamma_{v_0}$.
    The length of the path $\widetilde{\gamma}$ is bounded by $\ell_\calG(\gamma)$, hence $k \leq (1 + \ell_\calG(\gamma))^n = \norm{\gamma\Gamma_{v_0}}^\Gamma_n$ for $n \geq 1$.
    By~\cref{eq: norm bound geodesic Bass-Serre tree}, it follows that
    \begin{align*}
        \norm{s_{v_0}(\gamma\Gamma_{v_0})}^\Gamma_n \leq \sum_{i=1}^k \norm{\widetilde{e}_i}^\Gamma_n
        \leq k \cdot \norm{\gamma\Gamma_{v_0}}^\Gamma_n
        \leq \bigl(\norm{\gamma\Gamma_{v_0}}^\Gamma_n\bigr)^2,
    \end{align*}
    hence $s_{v_0}$ is $\sfB_\Gamma$\=/bounded.

    Now consider the general case of a vertex $\gamma\Gamma_v$. Denote by $\widetilde{Y}_v$ the Bass-Serre tree associated
    to $\pi_1(\calG, v)$. The path $p_v$ induces isomorphisms $p_{v,0}$ and $p_{v,1}$ such that
    \[\begin{tikzcd}
        {C_1(\widetilde{Y};\coeffZ)} & {C_0(\widetilde{Y};\coeffZ)} \\
        {C_1(\widetilde{Y}_v;\coeffZ)} & {C_0(\widetilde{Y}_v;\coeffZ)}
        \arrow["\partial", from=1-1, to=1-2]
        \arrow["{p_{v,1}}", from=1-1, to=2-1]
        \arrow["\cong"', draw=none, from=1-1, to=2-1]
        \arrow["\commutes"{description}, draw=none, from=1-1, to=2-2]
        \arrow["{p_{v,0}}", from=1-2, to=2-2]
        \arrow["\cong"', draw=none, from=1-2, to=2-2]
        \arrow["{\partial_v}", from=2-1, to=2-2]
    \end{tikzcd}\]
    commutes. Denote by $s_v$ the map $C_0(\widetilde{Y}_v;\coeffZ) \to C_1(\widetilde{Y}_v;\coeffZ)$ constructed
    in the same manner as $s_{v_0}$. We now define
    $s'(\gamma\Gamma_v) \coloneqq (p_{v,1}^{-1} \circ s_v \circ p_{v,0})(\gamma\Gamma_v)$
    such that $\partial(s'(\gamma\Gamma_v)) = \gamma\Gamma_v - 1\Gamma_v$.
    Each path $p_v$ lifts to a path $\widetilde{p}_v$ from $1\Gamma_{v_0}$ to $1\Gamma_{v}$ in $\widetilde{Y}$.
    Then the map $s(\gamma\Gamma_v) \coloneqq \widetilde{p}_v + s'(\gamma\Gamma_v)$ satisfies
    $\partial s(\gamma\Gamma_v) = \gamma\Gamma_v - 1\Gamma_{v_0}$, hence is a retraction of $\partial$.
    Because there are only finitely many vertices in $Y$, the length functions
    of $\pi_1(\calG,v)$ for $v \in V$ are all equivalent to $\pi_1(\calG,v_0)$ up to constants independent of the vertex $v$.
    In particular, the $\sfB_\Gamma$\=/boundedness of $s_{v_0}$ implies that $s_v$ and therefore $s'$ and $s$ are $\sfB_\Gamma$\=/bounded.
\end{proof}

\section{Application to \texorpdfstring{$S$}{S}-arithmetic lattices}\label{sec: S-arithmetic lattices}

Let $k$ be a number field. Let $\mathbf{G}$ be an absolutely almost simple, $k$-isotropic $k$-algebraic group.
For every place $\nu$ of $k$ let $k_\nu$ be the corresponding completion and set $G_\nu=\mathbf{G}(k_\nu)$.
Denote by $S_\infty$ the set of archimedean places of $k$, and let $S$ be a finite set of places of $k$ that contains $S_\infty$. Let $S_f=S\setminus S_\infty$. For every subset $S'\subset S$ we let $G_{S'}$ denote the group $G_{S'}=\prod_{\nu \in S'} G_\nu$.
We have
\[ G_S=G_{S_\infty}\times G_{S_f},\]
where $G_{S_\infty}$ is a real semisimple Lie group.

We set $\rank_S (\mathbf{G})=\sum_{\nu\in S} \rank_{k_\nu} \mathbf{G}$.
We let $\mathcal{O}\subset k$ be the ring of integers and $\mathcal{O}_S\subset k$ be the $S$-integers, that is, the set of elements $x\in k$ such that $\abs{x}_\nu\le 1$ for every non-archimedean valuation $\nu$ of $k$ that is not in~$S$.
The subgroup $\mathbf{G}(\mathcal{O}_S)<G_S$ is well-defined only up to commensurability.
Every subgroup $\Gamma<G_S$ that is commensurable with $\mathbf{G}(\mathcal{O}_S)$ is called an \emph{$S$-arithmetic lattice} of~$G_S$. In case $S=S_\infty$ we say that $\Gamma$ is an \emph{arithmetic lattice}.
We have a precise conjectural picture regarding the filling rank of an $S$-arithmetic lattice.

\begin{conj}\label{conj:BEW}
    Let $\Gamma<G_S$ be an $S$-arithmetic lattice as above. Then the filling function $\filling_{\bbZ,\Gamma}^i$ is polynomial for every $2\le i\le \rank_{S}(\mathbf{G})-1$.
\end{conj}

This was conjectured by Bestvina-Eskin-Wortman in~\cite{bestvina+eskin+wortman}*{Conjecture 3} for the homotopical filling functions but it follows from the above conjecture. In fact, they conjecture that the homotopical filling functions are Euclidean. A breakthrough was the confirmation of the conjecture for irreducible arithmetic lattices in semisimple Lie groups by Leuzinger-Young.

\begin{thm}[Leuzinger-Young~{\cite{leuzinger+young}}]%
    \label{thm:LY}
    Conjecture~\ref{conj:BEW} holds provided $S=S_\infty$. That is, it holds for $\Gamma$ an irreducible lattice in a connected semisimple Lie group with finite center and no compact factors.
\end{thm}

In the sequel, $\Gamma<G_S$ denotes an $S$-arithmetic lattice.
For $\nu\in S_\infty$, the group $G_\nu$ acts on its associated symmetric space $X_\nu\cong G_\nu/K_\nu$, which is a non-positively curved Riemannian manifold. We set
\[X_f=\prod_{\nu\in S_f} X_\nu,~X_\infty=\prod_{\nu\in S_\infty}X_\nu,~\text{and}~X=X_\infty\times X_f.\]
For $\nu\in S_f$, the group $G_\nu$ acts on its associated Bruhat-Tits building $X_\nu$, which is a \CAT{0}-space. We obtain product actions of $G_{S_\infty}$ on $X_\infty$ and of $G_{S_f}$ on $X_f$ and of $G_S$ on $X$.
Further, $G_S$, thus $\Gamma$, also acts on $X_f$ via the projection $G_S\to G_{S_f}$ and the product action of $G_{S_f}$ on~$X_f$.

\begin{lem}\label{lem: filling function of stabilizer}
    Let $\sigma$ be a cell of $X_f$,
    and let $\Gamma_\sigma$ be the stabilizer of the $\Gamma$-action on~$X_f$.
    Then the filling function $\filling^i_{\bbZ,\Gamma_\sigma}$ is polynomial for $2\le i\le \rank_{S_\infty}\bfG-1$.
\end{lem}

\begin{proof}
    The stabilizers of the $G_\nu$-action on $X_\nu$ for $\nu\in S_f$ are compact-open and hence commensurable with the compact-open subgroup $\bfG(\calO_\nu)$ of $G_\nu$.
    It follows that every stabilizer $\Gamma_\sigma$ of the $\Gamma$-action on~$X_f$ is commensurable with
    \[ \bfG(\calO_S)\cap \bigcap_{\nu\in S_f}\bfG(\calO_\nu)=\bfG(\calO).\]
    The subgroup $\bfG(\calO)$ is a lattice in the semisimple Lie group $G_{S_\infty}$, and we can apply~\cref{thm:LY} to it. Thus, the result follows.
\end{proof}

\begin{lem}\label{lem: stabilizer undistorted}
    Let $\sigma$ be a cell of $X_f$,
    and let $\Gamma_\sigma$ be the stabilizer of the $\Gamma$-action on~$X_f$. Let $\rank_{S_\infty}\bfG\ge 2$. Then the inclusion $\Gamma_\sigma\hookrightarrow\Gamma$ is undistorted.
\end{lem}

\begin{proof}
    It is enough to prove that $\bfG(\calO)\hookrightarrow\bfG(\calO_S)$ is undistorted because $\Gamma_\sigma$ is commensurable with $\bfG(\calO)$. See the proof of~\cref{lem: filling function of stabilizer}. For every archimedean valuation~$\nu$ of~$k$, let $d_\nu$ be
    a $G_\nu$-invariant Riemannian metric defined on the symmetric space~$X_\nu$. For every non-archimedean valuation~$\nu$ of~$k$, let $d_\nu$ be a $G_\nu$-invariant (combinatorial) metric on the Bruhat-Tits building. In both cases, a $G_\nu$-invariant metric is not unique, but it is unique up to quasi-isometry.

    We pick base points $x_\nu\in X_\nu$ for every $\nu\in S$. For $(g_\nu), (h_\nu)\in G_S$, we define a metric on $G_S$ by
    \[ d_X\bigl((g_\nu),(h_\nu)\bigr)=\sum_{\nu\in S}d_\nu(g_\nu x_\nu, h_\nu x_\nu).\]
    Similarly, by replacing $S$ with $S_\infty$ we obtain a metric $d_{X_\infty}$ on $G_{S_\infty}$.
    Apart from these metrics we consider the word metrics $d_{\bfG(\calO_S)}$ and $d_{\bfG(\calO)}$ of these groups. By a result of Lubotzky-Mozes-Raghunathan~\cite{lubotzky+mozes+raghunathan}*{Theorem~A}, the identity maps are quasi-isometries
    \begin{equation}\label{eq: QI arithmetic}
        \bigl(\bfG(\calO),d_{\bfG(\calO)}\bigr)\approx (\bfG(\calO), d_{X_\infty}\vert_{\bfG(\calO)\times \bfG(\calO)})
    \end{equation}
    and, similarly,
    \begin{equation}\label{eq: QI S-arithmetic}
        (\bfG(\calO_S),d_{\bfG(\calO_S)})\approx (\bfG(\calO_S), d_X\vert_{\bfG(\calO_S)\times \bfG(\calO_S)}).
    \end{equation}
    By extending a generating set of $\bfG(\calO)$ to $\bfG(\calO_S)$, we may and will assume $d_{\bfG(\calO_S)}\vert_{\bfG(\calO)\times\bfG(\calO)}\le d_{\bfG(\calO)}$. The reverse inequality up to constants follows from the quasi-isometries~\cref{eq: QI arithmetic,eq: QI S-arithmetic} and the obvious inequality $d_{X_\infty}\le d_{X}$.
\end{proof}

\begin{thm}\label{thm:filling functions for S-arithmetic groups}
    Let $\Gamma<G_S$ be an $S$-arithmetic lattice as above. Then the \Nth{i} filling function $\filling_{\bbZ,\Gamma}^i$ of $\Gamma$ is polynomial for every $i\in\{2,\dots, \rank_{S_\infty}\bfG-1\}$.
\end{thm}

\begin{proof}
    We may assume that $\rank_{S_\infty}\bfG>2$. Otherwise, the statement is void.
    We want to apply~\cref{thm: geometric brown theorem} to the $\Gamma$\=/action on the product of Bruhat-Tits buildings $X_f$. Upon passing to a barycentric subdivision, we may assume that $X_f$ is simplicial, which is required in~\cref{thm: geometric brown theorem}. It is well known that $X_f$ is a \CAT{0}-space~\cite{abramenko+brown}*{Theorem~11.16 on p.~555}, and $X_f$ satisfies the assumptions of~\cref{cor: CAT(0)}. Thus, its cellular chain complex $C_\ast(X_f;\bbZ)$ is an $\excat{\Gamma}$\=/resolution of the trivial module~$\bbZ$.

    By~\cref{lem: filling function of stabilizer}, the filling function $\filling_{\bbZ,\Gamma_\sigma}^i$ of every cell $\sigma$ in $X_f$ is polynomial in the range $2\le i\le \rank_{S_\infty}\bfG-1$. Furthermore, the inclusion $\Gamma_\sigma\hookrightarrow \Gamma$ is undistorted by~\cref{lem: stabilizer undistorted}. Therefore, the assumptions of~\cref{thm: geometric brown theorem} are satisfied, and its conclusion is that $\filling_{\bbZ,\Gamma}^i$ is polynomial for every $i\in\{2,\dots, \rank_{S_\infty}\bfG-1\}$.
\end{proof}

One should compare the theorem to the following theorem of Bestvina-Eskin-Wortman, which is proved in the same paper where they discuss~\cref{conj:BEW}.

\begin{thm}[\cite{bestvina+eskin+wortman}*{Corollary 5}]\label{thm: BEW}
    Let $\Gamma<G_S$ be an $S$-arithmetic lattice as above. Then the \Nth{i} filling function $\filling_{\bbZ,\Gamma}^i$ of $\Gamma$ is polynomial for every $i\in\{2,\dots, \size{S}-1\}$.
\end{thm}

Whether our~\cref{thm:filling functions for S-arithmetic groups} or that of Bestvina-Eskin-Wortman yields the stronger statement depends on how many primes are inverted in comparison with the rank. The following is a specific instance of~\cref{thm:filling functions for S-arithmetic groups}, where~\cref{thm: BEW} does not give any information.

\begin{thm}\label{thm:arith-special-application}
    For every prime~$p$,
    the \Nth{i} filling function $\filling_{\bbZ,\Gamma}^i$ of $\Gamma=\SL_n(\mathbb{Z}[1/p])$ is polynomial for every $i\in\{2,\dots, n-2\}$.
\end{thm}

Note that~\cref{conj:BEW} predicts that $\filling_{\bbZ,\Gamma}^i$ is polynomial in the range $2\le i\le 2n-3$.

\end{document}